\documentclass[a4paper]{amsart}
\usepackage{a4wide}
        \usepackage{latexsym}
        \usepackage{amssymb}
        \usepackage{amsmath}
        \usepackage{amsfonts}
        \usepackage{amsthm}
        \usepackage[colorlinks=true]{hyperref}
        \ifpdf
          \usepackage[all,knot,poly,2cell,cmtip]{xy}
        \else
          \usepackage[all,knot,poly,2cell,dvips,cmtip]{xy}
        \fi
             \UseAllTwocells
        \usepackage{xspace}

        \usepackage{eucal}

        \numberwithin{equation}{section}
        \theoremstyle{plain}
        \newtheorem{theorem}[equation]{Theorem}
        \newtheorem{corollary}[equation]{Corollary}
        \newtheorem{lemma}[equation]{Lemma}
        
        \newtheorem{proposition}[equation]{Proposition}
        \newtheorem{maintheorem}{Theorem}

        \theoremstyle{definition}
        \newtheorem{definition}[equation]{Definition}
        \newtheorem{example}[equation]{Example}
        
        \newtheorem*{example*}{Example}

        \theoremstyle{remark}
        \newtheorem{remark}[equation]{Remark}
        \newtheorem*{remark*}{Remark}

        \newcommand{\suchthat}{\,:\,}
        \newcommand{\itemref}[1]{\eqref{#1}}
        
        \newcommand{\loccit}{[\emph{loc.\ cit.}]\xspace}

        \newcommand{\C}{\mathbb{C}}
        \newcommand{\Z}{\mathbb{Z}}

        \newcommand{\Q}{\mathbb{Q}}

        \newcommand{\Orb}{\mathcal{O}}   
       \DeclareMathOperator{\spec}{Spec} 
           \newcommand{\et}{\mathrm{\acute{e}t}}

        \newcommand{\MOD}{\mathsf{Mod}}    

        \DeclareMathOperator{\im}{im} \DeclareMathOperator{\Hom}{Hom}

        \newcommand{\COHO}[1]{\mathcal{H}^{{#1}}}
        \newcommand{\trunc}[1]{\tau^{{#1}}}
        \newcommand{\RDERF}{\mathsf{R}}
        \newcommand{\LDERF}{\mathsf{L}}
        
        \newcommand{\DCAT}{\mathsf{D}}
        \newcommand{\RHom}{\RDERF\!\Hom}

        \newcommand{\QCOH}{\mathsf{QCoh}}
        \newcommand{\COH}{\mathsf{Coh}}

        \newcommand{\ID}[1]{\mathrm{Id}_{#1}}

        \newcommand{\tensor}{\otimes}

\renewcommand{\subset}{\subseteq}
\numberwithin{equation}{section}

\newcommand{\qcsubscript}{\mathrm{qc}} 

\newcommand{\DQCOH}[1][]{\DCAT_{\qcsubscript{#1}}} 

\newcommand{\holim}[1]{\underset{#1}{\mathrm{holim}}\,}

\newcommand{\spref}[1]{\href{http://stacks.math.columbia.edu/tag/#1}{#1}}

\makeatletter
\newcommand{\labitem}[2]{%
\def\@itemlabel{(\textbf{#1})}
\item
\def\@currentlabel{\textbf{#1}}\label{#2}}
\makeatother

\usepackage{mathtools}
\usepackage[inline]{enumitem}
\usepackage{amsmath}
\setlist[enumerate]{font=\normalfont}

\title{GAGA theorems}
\date{May 17, 2022}
\author[J. Hall]{Jack Hall}

\theoremstyle{remark}

\newtheorem*{claim*}{Claim}
\usepackage{stmaryrd}

\newcommand{\shfcoho}{\mathrm{H}}
\newcommand{\pCOH}{\mathrm{pc}}
\newcommand{\DPCOH}[2][]{\DCAT_{\pCOH{#1}}^{{#2}}}
\newcommand{\thick}[1]{\langle {#1} \rangle}
\newcommand{\Vect}{\mathsf{Vect}}

\renewcommand{\mathbb}{\mathbf}
\DeclareMathOperator{\cone}{cone}
\newcounter{Counter:Zequiv}
\newtheorem{setup}[equation]{Setup}
\makeatletter
\def\namedlabel#1#2{\begingroup
   \def\@currentlabel{#2}%
   \label{#1}\endgroup
}
\makeatother
\usepackage[T1]{fontenc}

\begin{document}
\begin{abstract}
  We prove a new and unified GAGA theorem. This recovers all analytic
and formal GAGA results in the literature, and is also valid in the
non-noetherian setting. Our method can also be used to
establish various Lefschetz theorems and comparison results for the Fargues--Fontaine curve. 

\end{abstract}
\maketitle
\section{Introduction}
Let $X$ be a scheme. Frequently associated to $X$ is a morphism 
\[
  c \colon \mathcal{X} \to X,
\]
where $\mathcal{X}$ is some type of analytic space. 
When $X$ is proper over $\spec R$, where $R$ is a noetherian ring (usually connected to the construction of $\mathcal{X}$), there is often an induced comparison isomorphism on cohomology of coherent sheaves:
\[
\shfcoho^i(X,F) \simeq \shfcoho^i(X,c^*F)
\]
and an equivalence of abelian categories of coherent sheaves:
\[
  c^* \colon \COH(X) \simeq \COH(\mathcal{X}).
\]
Since Serre's famous paper \cite{MR0082175}, such results have been
called ``GAGA theorems''. Restricting these comparison isomorphisms
and equivalences to specific subcategories of sheaves (e.g., vector
bundles or finite \'etale algebras) leads to both local and global ``Lefschetz
Theorems'' \cite{SGA2}. We briefly recall some examples of these phenomena
below.
\subsection{Archimedean analytification}
Assume $X$ is proper over $\spec \C$. Naturally associated to $X$ is an analytic space $X_{\mathrm{an}}$.  This
consists of endowing the $\C$-points of $X$ with its Euclidean
topology and its sheaf of holomorphic functions. There is an induced morphism of locally 
ringed spaces $c\colon X_{\mathrm{an}} \to X$. This is the setting of the original GAGA theorems \cite{MR0082175} and \cite[XII.4.4]{SGA1}. 
\subsection{Non-archimedean analytification}
Assume $X$ is proper over $\spec K$,
where $K$ is a complete non-archimedean field. Then there are various
analytifications associated to $X$: the Berkovich $X_{\mathrm{Berk}}$,
the adic $X_{\mathrm{adic}}$, and the rigid
$X_{\mathrm{rig}}$. There is also a GAGA theorem in this
context \cite{kopf_rigid_gaga} (see 
\cite{MR2266885} for a more recent account).
\subsection{Formal completion}
Let $Z =V(I)\subseteq X$ be a closed subscheme. Then there is the formal
completion $\hat{X}_{/Z}$. The locally ringed space
$\hat{X}_{/Z}$ has underlying topological space $Z$ and sheaf of
rings $\Orb_{\hat{X}_{/Z}} = \varprojlim_n \Orb_X/I^n$. If $X$ is proper over a complete noetherian ring, then there is a GAGA theorem \cite[III${}_1$]{EGA}. 
\subsection{Unification}
All of these results have previously been proved separately, though
their general strategies are very similar. First one treats 
projective spaces directly via a direct computation of the cohomology of line bundles, then a d\'evissage is performed using Chow's Lemma.

The main theorem of this article is that these GAGA results are true
much more generally and can be put into a single framework. Essentially all existing results
follow very easily from ours (see \S\ref{S:applications}). We state
one such result in the noetherian situation for locally ringed
$G$-spaces.
\begin{maintheorem}\label{MT:A}
  Let $R$ be a noetherian ring. Let $X \to \spec R$ be a proper
  morphism of schemes. Let $c\colon \mathcal{X} \to X$ be a morphism
  of locally ringed $G$-spaces. Let $X_{\mathrm{cl}}$ be the set of
  closed points of $X$ and let
  $\mathcal{X}_{\mathrm{cl},c} = c^{-1}(X_{\mathrm{cl}})$. Assume that
  \begin{enumerate}
  \item \label{GTI:A:coh}$\Orb_{\mathcal{X}}$ is coherent;
  \item\label{GTI:A:bd} if $\mathcal{F} \in \COH(\mathcal{X})$, then
    $\oplus_i \shfcoho^i(\mathcal{X},\mathcal{F})$ is a finitely
    generated $R$-module; 
  \item\label{GTI:A:bijective}
    $c\colon \mathcal{X}_{\mathrm{cl},c} \to X_{\mathrm{cl}}$ is
    bijective;
  \item\label{GTI:A:support} if $\mathcal{F} \in \COH(\mathcal{X})$
    and $\mathcal{F}_x=0$ for
    all $x\in \mathcal{X}_{\mathrm{cl},c}$, then $\mathcal{F} = 0$; and
  \item \label{GTI:A:flatunr_along_S} if $x\in \mathcal{X}_{\mathrm{cl},c}$, then
    $\Orb_{X,c(x)} \to \Orb_{\mathcal{X},x}$ is flat and
    $\kappa(c(x)) \to \kappa(c(x)) \tensor_{\Orb_{X,c(x)}}
    \Orb_{\mathcal{X},x}$ is an isomorphism.
  \end{enumerate}
  Then the comparison map:
  \[
    \shfcoho^{i}(X,F) \to \shfcoho^i(\mathcal{X}, c^*F)
  \]
  is an isomorphism for all coherent sheaves $F$ on $X$ and
  \[
    c^* \colon \COH(X) \to \COH(\mathcal{X})
  \]
  is an exact equivalence of abelian categories. 
\end{maintheorem}
In future work, we will apply  it to algebraic stacks and
their derived counterparts. Our approach is related
to the non-noetherian GAGA results in the Stacks Project
\cite[Tag \spref{0DIJ}]{stacks-project}.
While Theorem \ref{MT:A} (and \ref{T:lefschetz-dim1}, below) is an immediate consequence of much more general results
(Theorems \ref{T:general-lefschetz}, \ref{T:npGAGA}, and \ref{T:ncGAGA}), we give a simple and direct proof in \S\ref{S:direct}.
\begin{remark}\label{R:MT:A}
  Basic properties of noetherian local rings show that
  Theorem \ref{MT:A}\itemref{GTI:A:flatunr_along_S} is implied by:
  \begin{enumerate}
  \item[(\ref{GTI:A:flatunr_along_S}')] if
    $x\in \mathcal{X}_{\mathrm{cl},c}$, then $\Orb_{\mathcal{X},x}$ is
    noetherian and the morphism
    $\Orb_{X,c(x)} \to \Orb_{\mathcal{X},x}$ induces an isomorphism on
    maximal-adic completions.
  \end{enumerate}
\end{remark}
Conditions \itemref{GTI:A:coh} of \itemref{GTI:A:bd} of Theorem
\ref{MT:A} are typically established by theorems in analysis. For
example, in the archimedean case, they are Oka's Theorem
\cite{oka_coherence} and Cartan--Serre's finiteness theorem
\cite{cartan-serre_finiteness}. In the non-archimedian case, it is
Kiehl's finiteness theorem \cite{MR0210948}. In the case of formal
schemes, it is Grothendieck's finiteness theorem
\cite[III.3.4.2]{EGA}. Note that there is no circularity here: none of
these results have anything to do with GAGA theorems or their
proofs. It is even possible to use the analytic finiteness results to prove the algebraic ones \cite{MR3330765}. The recent work \cite{zavyalov2021coherent} shows that the
ideas in this article persist to almost mathematics. Theorem
\ref{MT:A} (and its generalizations considered in this paper) are
currently not well-adapted to local or henselian type analytifications
(e.g., \cite[Thm.~C.1.1]{abramovich_temkin_factorization_qe_char0} or
\cite{devadas-phd-henselian}), however, as the boundedness condition
\itemref{GTI:A:bd} is established via induction and GAGA.

Without finiteness results, however, we can establish the
following.
\begin{maintheorem}\label{T:lefschetz-dim1}
  Let $X$ be a quasi-compact and quasi-separated scheme. Let
  $c\colon \mathcal{X} \to X$ be a morphism of locally ringed
  $G$-spaces. Let $x\in X$ be a closed point. Assume
  \begin{enumerate}
  \item \label{TI:lefschetz-dim1:qaff} $U=X-\{x\}$ is quasi-affine;
  \item \label{TI:lefschetz-dim1:pt}$c^{-1}(x)$ consists of a single point $y$; 
  \item \label{TI:lefschetz-dim1:local}$\Orb_{X,x} \to \Orb_{\mathcal{X},y}$ is flat and
    $\kappa(x) \to \kappa(x) \tensor_{\Orb_{X,x}}
    \Orb_{\mathcal{X},y}$ is an isomorphism; and
  \item \label{TI:lefschetz-dim1:coho}
    $\shfcoho^0(X,\Orb_X) \simeq
    \shfcoho^0(\mathcal{X},\Orb_{\mathcal{X}})$ and
    $\shfcoho^1(X,\Orb_X) \hookrightarrow
    \shfcoho^1(\mathcal{X},\Orb_{\mathcal{X}})$.
  \end{enumerate}
  Then the comparison map:
  \[
  \shfcoho^0(X,E) \to \shfcoho^0(\mathcal{X},c^*E)	
  \]	
 is an isomorphism for all vector bundles $E$ of finite rank on $X$. 
\end{maintheorem}
Theorem \ref{T:lefschetz-dim1} is
sufficient to establish the full-faithfulness of the analytification
of vector bundles on the Fargues--Fontaine curve. Surprisingly, 
Theorem \ref{T:lefschetz-dim1} is a consequence of a general Lefschetz-style 
result (see \S\ref{S:lefschetz}), which establishes all existing full faithfulness results in the literature (see \S\ref{S:applications}).

In order to establish non-noetherian GAGA results in formal geometry
(e.g., \cite{fujiwara-kato-I} or \cite[Tag
\spref{0DIA}]{stacks-project}), the non-flatness of completions
complicates matters significantly. These are dealt with in
\S\ref{S:equivalences}. We work quite generally and expect these
methods to be applicable to other types of GAGA and Lefschetz-style
results (e.g., in o-minimal geometry). In Appendix
\S\ref{A:projection-formula}, we establish functoriality results for
the projection formula in symmetric monoidal categories, which turn
out to be critical for results like Theorem \ref{MT:A}.
\subsection*{Acknowledgements}
It is a pleasure to thank Amnon Neeman and David Rydh for several
supportive and useful discussions. We also wish to thank Bhargav Bhatt
and Brian Conrad for helpful comments on a
draft. We also wish to thank Rachel Webb for suggesting 
Lemma
\ref{L:projection-functorial-bc}\itemref{LI:projection-functorial-bc:projection}.
\section{Proof of Theorems \ref{MT:A} and \ref{T:lefschetz-dim1}}\label{S:direct}
We now prove Theorem \ref{MT:A} and \ref{T:lefschetz-dim1} in complete generality. Keeping track of the residue fields, while 
fussy, is the crucial calculation. 

We work with the unbounded derived category of $\Orb_X$-modules with
quasi-coherent cohomology sheaves, $\DQCOH(X)$, and the unbounded
derived category of $\Orb_{\mathcal{X}}$-modules,
$\DCAT(\mathcal{X})$. If $X$ is smooth and projective over a field,
then we can instead work with $\DCAT^b_{\COH}(X)$ and
$\DCAT^b_{\COH}(\mathcal{X})$. In general, the unbounded derived
category---while perhaps not strictly necessary---is much more convenient. There
is an unbounded derived pullback
\[	
	\LDERF c^* \colon \DQCOH(X) \to \DCAT(\mathcal{X}).
\]	
Throughout, we let $t\in \mathcal{X}$ and let $s=c(t)$. In the case of 
Theorem \ref{MT:A} we assume that $t\in \mathcal{X}_{\mathrm{cl},c}$ and in 
Theorem \ref{T:lefschetz-dim1} we take $t=y$. 
We now obtain a commutative diagram of ringed $G$-spaces:
\[
	\xymatrix{(\mathcal{X},\kappa(t)) \ar[r]^{c_s} \ar[d]_{i_{t}'} & (X,\kappa(s)) \ar[d]^{i_{s}} \\ \mathcal{X} \ar[r]_c  & X.}
\]
Note that $c_s^* \colon \MOD(X,\kappa(s)) \to \MOD(\mathcal{X},\kappa(t))$
is an exact equivalence of abelian categories because $\MOD(X,\kappa(s)) \simeq \MOD(\kappa(s))$, $\MOD(\mathcal{X},\kappa(t)) \simeq \MOD(\kappa(t))$, and $\kappa(s) \simeq \kappa(t)$ (see Lemma \ref{L:skyscraper-Gringed}). In particular, the natural morphism $\LDERF c^*\kappa(s) \to \kappa(t)$ is an isomorphism.

Now consider the unbounded derived pushforward
\[
	\RDERF c_* \colon \DCAT(\mathcal{X}) \to \DCAT(X),
\]	
which is right adjoint to $\LDERF c^*$. Then $t$ and $s$ are closed points, so $(i_s')_*$ and $(i_s)_*$ are exact. Also, $(c_s)_*$ is exact. In particular, 
$\RDERF c_*\kappa(t) \simeq \RDERF c_*(i_s')_*\kappa(t) \simeq (i_s)_*(c_s)_*\kappa(t) \simeq \kappa(s)$.	
Putting this all together, we have now established that the relevant adjunctions induce isomorphisms:
\begin{align}
	\kappa(s) &\simeq \RDERF c_*\LDERF c^*\kappa(s) & \LDERF c^*\RDERF c_{*}\kappa(t) &\simeq \kappa(t).\label{eq:s}
\end{align}
By general adjoint functor theorems, the inclusion $\DQCOH(X) \subseteq \DCAT(X)$ admits a right adjoint, 
\[
\RDERF Q_X \colon \DCAT(X) \to \DQCOH(X).
\]
Since $\DQCOH(X) \subseteq \DCAT(X)$ is a fully faithful embedding, if $M \in \DQCOH(X)$, then $M \simeq \RDERF Q_X(M)$. Composing $\RDERF Q_X$ with $\RDERF c_*$ above gives a right adjoint
\[
	\RDERF c_{\qcsubscript,*} \colon \DCAT(\mathcal{X}) \to \DQCOH(X),
\]
see \S\ref{S:adjoints} for a more in depth discussion. Thus, we can upgrade \eqref{eq:s} to
\begin{align}
	\kappa(s) &\simeq \RDERF c_{\qcsubscript,*}\LDERF c^*\kappa(s) 
	& \LDERF c^*\RDERF c_{\qcsubscript,*}\kappa(t) &\simeq \kappa(t) \label{eq:qcs}.
\end{align}
Let $N \in \DQCOH(X)$ and consider the adjunction
\[
 \eta_N \colon N \to \RDERF c_{\qcsubscript,*}\LDERF c^*N.
\]	
The subcategory $\mathcal{T}$ of $\DQCOH(X)$ with 
objects those $N$ 
for which $\eta_N$ is an isomorphism is certainly thick and 
triangulated. By \eqref{eq:qcs}, it contains $\kappa(s)$. A short inductive 
argument shows that $\mathcal{T}$ contains 
every $N \in \DCAT^b_{\COH}(X)$ such that $\LDERF j_s^*N \simeq 0$, where $j_s \colon X-\{s\} \subseteq X$ (see Lemma \ref{L:good-perfect}). 

We now make another key observation: if $P \in \DQCOH(X)$ is perfect and $\mathcal{N} \in \DCAT(\mathcal{X})$, then there is the \emph{projection formula}:
\begin{equation}
	P \tensor^{\LDERF}_{\Orb_X} (\RDERF c_{\qcsubscript,*}\mathcal{N}) \simeq \RDERF c_{\qcsubscript,*}(\LDERF c^*P \tensor^{\LDERF}_{\Orb_{\mathcal{X}}} \mathcal{N}). \label{eq:projection_special}
\end{equation}
This follows from formal properties of adjoints (see Lemma \ref{L:projection_formula} and Appendix \ref{A:projection-formula}). 

By perfect approximation \cite{MR2346195}, there is a perfect complex $P_s$ 
supported only at $s$ and a morphism $\phi_s \colon P_s \to \kappa(s)$ 
such that $\trunc{\geq 0}\phi_s$ is an isomorphism. But $P_s 
\tensor^{\LDERF}_{\Orb_X} N$ is supported only at $s$ too, so the projection formula implies that 
\begin{equation}
	P_s 
\tensor^{\LDERF}_{\Orb_X} H_N \simeq H_{P_s \tensor^{\LDERF}_{\Orb_X} N} \simeq 0 \label{eq:ps},
\end{equation} 
where $H_N$ is the cone of $\eta_N$ (see \eqref{LEQ:key-projection-ff:2}). The proofs of Theorem \ref{MT:A} and \ref{T:lefschetz-dim1} now diverge. 
\begin{proof}[Proof of Theorem \ref{MT:A}]
First, observe that $\LDERF c^*$ restricts to a $t$-exact functor on 
$\DCAT^-_{\COH}(X)$. Indeed, if $M \in \DCAT^-_{\COH}(X)$, then
  $(\LDERF c^*M)_t \simeq M_{s} \otimes^{\LDERF}_{\Orb_{X,s}}
  \Orb_{\mathcal{X},t}$. In particular, \itemref{GTI:A:flatunr_along_S} implies that if $\trunc{<n}M \simeq 0$ for some $n$, then
  \[
  (\trunc{<n}\LDERF c^*M)_t \simeq \trunc{<n}(M_{s} \otimes^{\LDERF}_{\Orb_{X,s}} \Orb_{\mathcal{X},t})\simeq (\trunc{<n}M_{s}) \tensor^{\LDERF}_{\Orb_{X,s}} \Orb_{\mathcal{X},t} \simeq 0.
  \]	 
	Condition \itemref{GTI:A:support} gives
  $\trunc{<n}(\LDERF c^*M) \simeq 0$ as claimed.
  
  Condition \itemref{GTI:A:bd} now allows us to apply a deep
  finiteness result, which implies that the restriction of
  $\RDERF c_{\qcsubscript,*}$ to $\DCAT^b_{\COH}(\mathcal{X})$ factors
  through $\DCAT^b_{\COH}(X)$ (see Proposition \ref{P:adjoints} or \cite[Prop.~3.0.9]{BZNP_integral}). If
  $X$ is projective over a field $k$, this step also follows from
  \cite[Rem.~4.6]{MR2793026} (if $X$ is smooth, it is the well-known
  \cite[Thm.~A.1]{MR1996800}). In summary, our functors restrict to an
  adjoint pair
\[
\LDERF c^* \colon \DCAT_{\COH}^b(X) \leftrightarrows \DCAT_{\COH}^b(\mathcal{X}) \colon \RDERF c_{\qcsubscript,*}.
\]
Note that we now have $H_N \in \DCAT^b_{\COH}(X)$. Let $v$ be an 
integer such that $\trunc{>v}H_N = 0$; then $\kappa(s) 
\tensor_{\Orb_X} \COHO{v}(H_N)\simeq \COHO{v}(P_s \tensor^{\LDERF}_{\Orb_X} H_N) 
\simeq 0$, by \eqref{eq:ps}. Since this is true for all closed points $s\in |X|$ 
and $\COHO{v}(H_N) \in \COH(X)$, Nakayama's Lemma implies that $\trunc{>v-1}H_N \simeq 0$. It follows immediately that $H_N \simeq 0$. 
Hence, $\eta_N$ is an isomorphism for all $N\in 
\DCAT^b_{\COH}(X)$ and so $\LDERF c^*$ is fully faithful on $\DCAT^b_{\COH}(X)$. Further, we 
have the comparison result: if $N \in \COH(X)$ and $i\geq 0$, then
\begin{align*}
  \shfcoho^i(X,N) &\simeq \Hom_{\Orb_X}(\Orb_X,N[i]) \simeq \Hom_{\Orb_X}(\Orb_X,\RDERF c_{\qcsubscript,*}c^*N[i]) \\
  &\simeq \Hom_{\Orb_{\mathcal{X}}}(c^*\Orb_X,c^*N[i]) \simeq \shfcoho^i(\mathcal{X},c^*N).
\end{align*}
For the essential surjectivity of $\LDERF c^*$, we proceed similarly as 
before: let $\mathcal{M} \in \DCAT^b_{\COH}(\mathcal{X})$; then we must prove 
that $\epsilon_{\mathcal{M}} \colon \LDERF c^*\RDERF c_{\qcsubscript,*}
\mathcal{M} \to \mathcal{M}$ is an isomorphism. Let $E_{\mathcal{M}}$ be 
the cone of $\epsilon_{\mathcal{M}}$; then $E_{\mathcal{M}} \in \DCAT^b_{\COH}
(\mathcal{X})$. The projection formula and \eqref{eq:qcs} give
	$\LDERF c^*P_s \tensor_{\Orb_{\mathcal{X}}} E_{\mathcal{M}} \simeq E_{\LDERF 
	c^*P_s \tensor_{\Orb_{\mathcal{X}}} \mathcal{M}} \simeq 0$. 	
Condition \itemref{GTI:A:support} now shows that $E_{\mathcal{M}} \simeq 0$.
\end{proof} 
\begin{remark}\label{R:necessity}
  Let $c\colon \mathcal{X} \to X$ be a morphism of
  locally ringed $G$-spaces  $\Orb_{\mathcal{X}}$ is coherent and such that
  $\LDERF c^* \colon \DCAT^b_{\COH}(X) \to
  \DCAT^b_{\COH}(\mathcal{X})$ is a $t$-exact equivalence. If there is a set of closed points $S$ of $\mathcal{X}$ such that
  \begin{enumerate}
  \item[(a)] if $s\in S$, then $\Orb_{\mathcal{X},s}$ is noetherian; 
  \item[(b)] if $s\in S$, then $\kappa(s) \in \COH(\mathcal{X})$; and
  \item[(c)] if $\mathcal{F} \in \COH(\mathcal{F})$ and $\mathcal{F}_s =0$
    for all $s\in S$, then $\mathcal{F} = 0$.
  \end{enumerate}
   Then a short argument gives $S=c^{-1}(X_{\mathrm{cl}})$ and the remaining conditions of
  Theorem \ref{MT:A} are met. 
\end{remark}

\begin{proof}[Proof of Theorem \ref{T:lefschetz-dim1}]
  It suffices to prove that if $F$ is
  a vector bundle on $X$, then $\trunc{\leq 0}H_F \simeq 0$. Indeed, there is an exact sequence:
  \[
\xymatrix{\shfcoho^{-1}(X,H_F) \ar[r] & \shfcoho^0(X,F) \ar[r] & \shfcoho^0(X,\RDERF c_{\qcsubscript,*} c^*F) \ar[r] & \shfcoho^0(X,H_F)}
  \]
  If $\trunc{\leq 0}H_F \simeq 0$, then the terms at the ends vanish
  and by adjunction
  $\shfcoho^0(X,\RDERF c_{\qcsubscript,*}c^*F) \simeq
  \shfcoho^0(\mathcal{X},c^*F)$.
  
  Now since $F$ is a vector bundle, the projection formula  \eqref{eq:projection_special} implies that 
   $H_{\Orb_X} \tensor^{\LDERF}_{\Orb_X} F \simeq H_F$. Set
  $H=H_{\Orb_X}$; hence, it suffices to prove that
  $\trunc{\leq 0}H \simeq 0$. By \eqref{eq:ps},
  $P_s \tensor^{\LDERF}_{\Orb_X} H \simeq 0$. Now let
  $j\colon U \hookrightarrow X$ be the resulting open immersion that is complementary to $\{s\}$; then
  localization theory (e.g., \cite[Ex.~1.4]{telescope-stacks}) now
  implies that $H\simeq \RDERF j_*\LDERF j^*H$. But \itemref{TI:lefschetz-cartier:qaff} implies that $U$ is quasi-affine, so $\trunc{\leq 0}H\simeq 0$ if and
  only if $\trunc{\leq 0}\RDERF \Gamma(X,H) \simeq 0$. We now have the
  long exact sequence:
  \[
    0 \to \shfcoho^{-1}(X,H) \to \shfcoho^0(X,\Orb_X) \to
    \shfcoho^0(\mathcal{X},\Orb_X) \to \shfcoho^0(X,H) \to
    \shfcoho^1(X,\Orb_X) \to
    \shfcoho^1(\mathcal{X},\Orb_{\mathcal{X}}).
  \]
  Certainly, $\trunc{<-1}H \simeq 0$, so we have
  $\shfcoho^{-1}(X,H) \simeq \shfcoho^0(X,\COHO{-1}(H))$. But
  $\shfcoho^0(X,\Orb_X) \simeq
  \shfcoho^0(\mathcal{X},\Orb_{\mathcal{X}})$, so
  $\shfcoho^0(X,\COHO{-1}(H)) = 0$ and $\COHO{-1}(H) \simeq 0$. Hence,
  $\shfcoho^{0}(X,H) \simeq \shfcoho^0(X,\COHO{0}(H))$. The injection
  $\shfcoho^1(X,\Orb_X) \to
  \shfcoho^1(\mathcal{X},\Orb_{\mathcal{X}})$ allows us to conclude
  that $\COHO{0}(H) \simeq 0$ too. The result follows.
\end{proof}
\begin{remark}
  A support theory for ``big'' objects in $\DCAT(\mathcal{X})$, or
  some suitable subcategory, would aid in establishing the
  essential surjectivity of $c^*$ in Theorem \ref{T:lefschetz-dim1}.
\end{remark}
\begin{lemma}\label{L:skyscraper-Gringed}
  Let $\mathcal{Y}$ be a $G$-space. Let $B$ be a ring. Let
  $i\colon \{y\} \subseteq \mathcal{Y}$ be the inclusion of a
  point. If $y$ belongs to an admissible open and $\mathcal{Y}-\{y\}$
  is covered by admissible opens, then
  \[
    i^{-1} \colon \MOD(\mathcal{Y},i_*B) \leftrightarrows \MOD(B)
    \colon i_*
  \]
  is an exact equivalence of abelian categories.
\end{lemma}
\begin{proof}
  Let $\mathcal{F}$ be an $i_*B$-module. Let $V \subseteq \mathcal{Y}$
  be admissible. If $y\notin V$, then $(i_*B)(V) = 0$. But
  $\mathcal{F}(V)$ is a $(i_*B)(V)$-module, so $\mathcal{F}(V) =
  0$. Let $U \subseteq V$ be another admissible open and assume that
  $y\in U$. Let $\{W_i \subseteq \mathcal{Y}-\{y\}\}_{i\in I}$ be an
  admissible cover. The sheaf condition gives an exact sequence:
  \[
    \xymatrix@C-0.7pc{ 0 \ar[r] & \mathcal{F}(V) \ar[r] & \mathcal{F}(U)
      \times \prod_{i\in I} \mathcal{F}(V \cap W_i) \ar@<+0.5ex>[r]
      \ar@<-0.5ex>[r] & \left(\prod_{i\in I} \mathcal{F}(U \cap V \cap
        W_i)\right) \times \left(\prod_{i,j\in I} \mathcal{F}(V \cap
        W_i \cap W_j)\right).}
  \]
  Since $y\notin V \cap W_i$ for all $i\in I$, it follows that the
  sequence above collapses to the restriction morphism
  $\mathcal{F}(V) \to \mathcal{F}(U)$ being an isomorphism.

  Since $i^{-1}\mathcal{F} = \mathcal{F}_y$, it follows immediately
  that the adjunction $\mathcal{F} \to i_*i^{-1}\mathcal{F}$ is an
  isomorphism of abelian sheaves.  Now let $M$ be a $B$-module. Since
  it is clear that $i^{-1}i_*M \to M$ is an isomorphism of
  $B$-modules, the result follows.
\end{proof}
\section{A finiteness result}
Our first task is to consider a variant of the finiteness
result \cite[Thm.~1.1]{MR1996800} for non-noetherian algebraic
spaces. This was recently established in the noetherian case in
\cite{BZNP_integral} and in general in \cite{stacks-project}, where it
is formulated in terms of pseudo-coherence \cite{MR0354655}. Since the
non-noetherian situation will be important to us, we will briefly recall these ideas. 

Let $B$ be a ring. A bounded complex of finitely generated and
projective $B$-modules is called \emph{strictly perfect}. Let
$m\in \Z$. A complex of $B$-modules $M$ is \emph{$m$-pseudo-coherent}
if there is a morphism $\phi\colon P \to M$ such that $P$ is strictly
perfect and the induced morphism
$\COHO{i}(\phi) \colon \COHO{i}(P) \to \COHO{i}(P)$ is an isomorphism
for $i>m$ and surjective for $i=m$. A complex of $B$-modules $M$ is
\emph{pseudo-coherent} if it is $m$-pseudo-coherent for all integers
$m\in \Z$; equivalently, it is quasi-isomorphic to a bounded above
complex of finitely generated and projective $B$-modules \cite[Tag
\spref{064T}]{stacks-project}. These conditions are all stable under
derived base change \cite[Tag \spref{0650}]{stacks-project} and are flat local
\cite[Tag \spref{068R}]{stacks-project}.

We let $\DPCOH{-}(B)$ denote
the full triangulated subcategory of the derived category of
$B$-modules, $\DCAT(B)$, with objects those complexes of $B$-modules
that are quasi-isomorphic to a pseudo-coherent complex of $B$-modules.
We let $\DPCOH{b}(B) \subseteq \DPCOH{-}(B)$ be the triangulated
subcategory of objects with bounded cohomological support. If
$B \to C$ is a ring homomorphism, then derived base change induces
$-\tensor^{\LDERF}_B C\colon \DPCOH{-}(B) \to \DPCOH{-}(C)$. If $C$
has finite tor-dimension over $B$ (e.g., $C$ is $B$-flat), then the
derived base change sends bounded pseudo-coherent complexes to bounded
pseudo-coherent complexes.

The above generalizes to ringed sites \cite[Tag
\spref{08FS}]{stacks-project}. Let $\mathcal{X}$ be a ringed site. A
complex of $\Orb_{\mathcal{X}}$-modules is \emph{strictly perfect} if
it is bounded and each term is a direct summand of a finitely
generated and free $\Orb_{\mathcal{X}}$-module \cite[Tag
\spref{08FL}]{stacks-project}. A complex of
$\Orb_{\mathcal{X}}$-modules is \emph{perfect} if locally on
$\mathcal{X}$ is strictly perfect.
\begin{example}\label{E:perfect_support}
  Let $X$ be an algebraic space. If $i \colon D \subseteq X$ is a
  Cartier divisor, then $i_*\Orb_D \in \DQCOH(X)$ is perfect. More
  generally, if $i \colon Z \hookrightarrow X$ is a regular embedding
  (i.e., $i$ is locally the zero locus of a regular section of a
  vector bundle), then $i_*\Orb_Z \in \DQCOH(X)$ is perfect. Also, if
  $X$ is quasi-compact and quasi-separated and
  $j \colon U \subseteq X$ is a quasi-compact open immersion, then
  there is a perfect complex $P \in \DQCOH(X)$ whose cohomological
  support is precisely $X\setminus U$
  \cite[Thm.~A]{perfect_complexes_stacks}.
\end{example}

Let $m\in \Z$. A complex of
$\Orb_{\mathcal{X}}$-modules $\mathcal{M}$ is
\emph{$m$-pseudo-coherent} if locally on $\mathcal{X}$ there is a
morphism $\phi \colon \mathcal{P} \to \mathcal{M}$ such that
$\mathcal{P}$ is strictly perfect and the induced morphism
$\COHO{i}(\phi) \colon \COHO{i}(\mathcal{P}) \to
\COHO{i}(\mathcal{M})$ is an isomorphism for $i>m$ and surjective for
$i=m$. A complex of $\Orb_{\mathcal{X}}$-modules is
\emph{pseudo-coherent} if it is $m$-pseudo-coherent for every
$m\in \Z$.

Let $\DPCOH{-}(\mathcal{X})$ denote the full triangulated subcategory
of $\DCAT(\mathcal{X})$, the unbounded derived category of
$\Orb_{\mathcal{X}}$-modules, with objects those complexes that are
quasi-isomorphic to a bounded above pseudo-coherent complex. We let
$\DPCOH{b}(\mathcal{X}) \subseteq \DPCOH{-}(\mathcal{X})$ be the full
triangulated subcategory of objects with bounded cohomological
support. If $c \colon \mathcal{X} \to X$ is a morphism of ringed sites, then the restriction of $\LDERF c^* \colon \DCAT(X) \to \DCAT(\mathcal{X})$ to $\DPCOH{-}(X)$ factors through $\DPCOH{-}(\mathcal{X})$ \cite[Tag \spref{08H4}]{stacks-project}. Moreover, if $c$ has finite
tor-dimension (e.g., it is flat), then $\LDERF c^*$ preserves bounded complexes.
\begin{example}
  Perfect complexes are pseudo-coherent. In particular, vector bundles
  of finite rank are pseudo-coherent.
\end{example}
\begin{example}
  Let $\mathcal{X}$ be a ringed site with a coherent structure
  sheaf. For example, a locally noetherian algebraic space or an
  analytic space. Let $*\in \{-,b\}$. Then
  $\DPCOH{*}(\mathcal{X}) = \DCAT_{\COH}^*(\mathcal{X})$; that
  is, a complex $M \in \DCAT(\mathcal{X})$ is pseudo-coherent if and only if it
  is quasi-isomorphic to a bounded above complex of sheaves with
  coherent cohomology \cite[Cor.~I.3.5]{MR0354655}.
\end{example}
The following lemma improves upon those given in Example
\ref{E:perfect_support} in the coherent situation. To state this
lemma, we recall the following definition
\cite[\S2.1]{MR1812507}. Let $\mathcal{T}$ be a
triangulated category. A subcategory
$\mathcal{S} \subseteq \mathcal{T}$ is \emph{thick} (or \'epaisse) if
it is triangulated and is closed under direct summands.  If
$S \subseteq \mathcal{T}$ is a collection of objects, we let
$\thick{S} \subseteq \mathcal{T}$ denote the thick closure of $S$;
that is, it is the smallest thick triangulated subcategory of
$\mathcal{T}$ containing $S$.
\begin{lemma}\label{L:good-perfect}
  Let $X$ be a quasi-compact and quasi-separated algebraic space. Let
  $i\colon Z \hookrightarrow X$ be a finitely presented closed
  immersion. Let
  \[
    \DPCOH[,|Z|]{b}(X) = \ker(\LDERF j^* \colon
    \DPCOH{b}(X) \to \DPCOH{b}(U)),
  \]
  where $j \colon U \to X$ is the open immersion complementary to
  $i \colon Z \to X$. If $\Orb_X$ is coherent, then
  $\thick{\RDERF i_* \DPCOH{b}(Z)} = \DPCOH[,|Z|]{b}(X)$.
\end{lemma}
\begin{proof}
  Clearly, $\RDERF i_*\DPCOH{b}(Z) \subseteq
  \DPCOH[,|Z|]{b}(X)$. Since $\DPCOH[,|Z|]{b}(X)$ is a
  thick subcategory of $\DPCOH{b}(X)$, it follows that
  $\thick{\RDERF i_*\DPCOH{b}(Z)} \subseteq
  \DPCOH[,|Z|]{b}(X)$. For the reverse inclusion, by induction on the length of a complex, it is
  sufficient to prove that if
  $M \in \COH_{|Z|}(X)=\ker(j^*\colon \COH(X) \to \COH(U))$ then
  $M \in \thick{\RDERF i_*\DPCOH{b}(Z)}$. Let
  $I = \ker (\Orb_X \to i_*\Orb_Z)$. By
  \cite[Lem.~2.5(i)]{mayer-vietoris}, it follows that there exists an
  integer $n>0$ such that $I^{n+1}M = 0$. Hence, $M$ admits a finite
  filtration by $i_*\Orb_Z$-modules and so belongs to
  $\thick{\RDERF i_*\DPCOH{b}(Z)}$. 
\end{proof}
Let $A$ be a ring. A $B$-algebra $A$ is \emph{pseudo-coherent} if it
admits a surjection from a polynomial ring
$\phi \colon A[x_1,\dots,x_n] \twoheadrightarrow B$ such that $B$ is a
pseudo-coherent $A[x_1,\dots,x_n]$-module. Pseudo-coherence is stable
under flat base change on $A$ and is \'etale local on $B$. See
\cite[Tag \spref{067X}]{stacks-project} for more background
material. This definition generalizes readily to morphisms of
algebraic spaces \cite[Tag \spref{06BQ}]{stacks-project}. We now
recall some examples that will be important to us.
\begin{example}
  Let $A$ be a noetherian ring. If $X \to \spec A$ is a locally of
  finite type morphism of algebraic spaces, then it is pseudo-coherent
  \cite[Tag \spref{06BX}]{stacks-project}.
\end{example}
\begin{example}
  Let $A$ be a ring. If $X \to \spec A$ is a flat and locally of
  finite presentation morphism of algebraic spaces, then it is
  pseudo-coherent \cite[Tag \spref{06BV}]{stacks-project}.
\end{example}
\begin{example}\label{E:cohesive}
  Let $A$ be a universally cohesive ring. That is, every finitely
  presented $A$-algebra is a coherent ring. The standard example is an
  $a$-adically complete valuation ring; for example, $A=\Orb_{\C_p}$, the ring of integers in the $p$-adically completed algebraic closure of $\Q_p$, $\C_p$. If $X \to \spec A$ is a
  locally of finite presentation morphism of algebraic spaces, then it
  is pseudo-coherent. This is the setting for Fujiwara--Kato's
  formalism of rigid geometry \cite{fujiwara-kato-I}.
\end{example}
The main result of this section is the following small refinement of
\cite[Tag~\spref{0CTT}]{stacks-project}. 
\begin{theorem}\label{T:finiteness}
  Let $A$ be a ring. Let $X \to \spec A$ be a quasi-compact,
  separated, and pseudo-coherent morphism of algebraic spaces. Let
  $M \in \DQCOH(X)$. If $\RHom_{\Orb_X}(P,M) \in \DCAT(A)$ is pseudo-coherent
  (pseudo-coherent and bounded) for all perfect complexes $P$, then
  $M$ is pseudo-coherent (pseudo-coherent and bounded).
\end{theorem}
\begin{proof}
  It is sufficient to prove that
  $\RDERF \Gamma(X,E \tensor^{\LDERF}_{\Orb_X} M) \in \DCAT(A)$ is
  pseudo-coherent for all pseudo-coherent $E$ on $X$. Indeed,
  pseudo-coherent morphisms are locally finitely presented, so $M$ is
  pseudo-coherent relative to $A$ \cite[Tag
  \spref{0CTT}]{stacks-project}. Since $X \to \spec A$ is
  pseudo-coherent, $M$ is pseudo-coherent on $X$ \cite[Tag
  \spref{0DHQ}]{stacks-project}. The boundedness result is
  \cite[Lem.~3.0.14]{BZNP_integral}---also see \cite[Tag
  \spref{0GFE}]{stacks-project}.

  We begin by observing that if $P \in \DQCOH(X)$ is perfect, then
  $\RHom_{\Orb_X}(P^\vee,M) \simeq \RDERF
  \Gamma(X,P\tensor^{\LDERF}_{\Orb_X}M) \in \DCAT(A)$, which is
  pseudo-coherent by assumption. Now there exists an integer $n$ such that
  $\shfcoho^r(X,N) = 0$ for all $r>n$ and $N \in \QCOH(X)$  \cite[Tag
  \spref{073G}]{stacks-project}. By
  \cite[Rem.~2.1.11]{MR2434692},
  $\trunc{\geq m}\RDERF \Gamma(X,G) \simeq \trunc{\geq
    m}\RDERF\Gamma(X,\trunc{\geq l}G)$ for all $m\in \Z$, $l\leq m-n$, and
  $G\in \DQCOH(X)$. Let $E\in \DQCOH(X)$ be pseudo-coherent and fix
  $m\in \Z$. Let $j\in \Z$ be such that $\trunc{>j}M \simeq 0$ and
  $\trunc{>j}E \simeq 0$. Choose a perfect complex
  $P \in \DQCOH(X)$ and morphism $\phi \colon P \to E$ such that
  $\trunc{\geq a}\mathrm{cone}(\phi) \simeq 0$, where $a=m-n-j$ \cite[Tag
  \spref{08HP}]{stacks-project}. Then
  \begin{align*}
    \trunc{\geq m}\RDERF \Gamma(X, E\tensor^{\LDERF}_{\Orb_X} M) &\simeq \trunc{\geq m}\RDERF\Gamma (X,\trunc{\geq m-n}(E \tensor^{\LDERF}_{\Orb_X}M))\\
                                                                 &\simeq \trunc{\geq m}\RDERF\Gamma (X,\trunc{\geq m-n}((\trunc{\geq m-n-j}E) \tensor^{\LDERF}_{\Orb_X}M))\\
                                                                 &\simeq \trunc{\geq m}\RDERF \Gamma(X, \trunc{\geq m-n}((\trunc{\geq m-n-j}P) \tensor^{\LDERF}_{\Orb_X}M)\\
                                                                 &\simeq \trunc{\geq m}\RDERF\Gamma (X,\trunc{\geq m-n}(P \tensor^{\LDERF}_{\Orb_X}M))\\
    &\simeq \trunc{\geq m}\RDERF \Gamma(X,P\tensor^{\LDERF}_{\Orb_X} M).
  \end{align*}
  We have already seen that
  $\RDERF \Gamma(X,P\tensor^{\LDERF}_{\Orb_X} M)\in \DCAT(A)$ is
  pseudo-coherent, and the claim follows. 
\end{proof}
\begin{remark}
  Theorem \ref{T:finiteness} has a converse if $X \to \spec A$ is
  proper. If $M$ is pseudo-coherent (resp.~pseudo-coherent and
  bounded) and $P$ is perfect, then
  $\RHom_{\Orb_X}(P,M) \simeq \RDERF \Gamma(X,P^\vee
  \tensor^{\LDERF}_{\Orb_X} M)$. Replacing $M$ by
  $P^\vee \tensor^{\LDERF}_{\Orb_X} M$, it suffices to prove that
  $\RDERF \Gamma(X,-)$ sends pseudo-coherent complexes to
  pseudo-coherent complexes. If $A$ is noetherian, this is just the
  usual coherence theorem for algebraic spaces
  \cite[Thm.~IV.4.1]{MR0302647}. If $X$ is a scheme and $A$ is not
  necessarily noetherian, this is Kiehl's Finiteness Theorem
  \cite[Thm.~2.9]{MR0382280}. If $X \to \spec A$ is flat, this is in the Stacks
  Project \cite[Tag \spref{0CSC}]{stacks-project}. If $A$ is
  universally cohesive, then this is due to Fujiwara--Kato
  \cite[Thm.~I.8.1.2]{fujiwara-kato-I}. Using derived algebraic
  geometry, the argument given in the Stacks Project readily extends
  to the general (i.e., non-flat) situation; that is, a version of
  Kiehl's finiteness theorem for algebraic spaces. This is done by
  Lurie in \cite{lurie_sag}.
\end{remark}
As noted in \cite[Rem.~3.0.6]{BZNP_integral}, it is Theorem
\ref{T:finiteness} that fails miserably for algebraic stacks with
infinite stabilizers. In future work, we will describe a variant of
Theorem \ref{T:finiteness} for a large class of algebraic stacks with infinite stabilizers that is sufficient to establish integral transform and GAGA results.

We conclude this section with a simple corollary of Theorem
\ref{T:finiteness}. Variants of this are well-known (see
\cite[Ex.~4.3.9]{MR2490557} and \cite{MO_coh_proper} in the finite
type noetherian, but non-separated situation). 
\begin{corollary}\label{C:bounded_coho_gives_proper}
  Let $A$ be a universally cohesive ring (e.g., noetherian). Let
  $X \to \spec A$ be a quasi-compact and separated morphism of
  algebraic spaces. If
  $\RDERF \Gamma(X,-)$ sends $\DPCOH{b}(X)$ to
  $\DPCOH{b}(A)$, then $X \to \spec A$ is proper and of finite
  presentation.
\end{corollary}
\begin{proof}
  By absolute noetherian approximation \cite{rydh-2009}, there is an
  affine morphism $a\colon X\to X_0$, where $X_0$ is a separated and
  finitely presented algebraic space over $\spec A$. Using Nagata's
  compactification theorem for algebraic spaces \cite{MR2979821}, a
  blow-up, and absolute noetherian approximation again, we may further
  assume that $X_0 \to \spec A$ is proper and finitely
  presented. Since $A$ is universally cohesive, $\Orb_{X_0}$ is
  coherent. Now let $P \in \DQCOH(X_0)$ be a perfect complex; then
  \[
    \RHom_{\Orb_{X_0}}(P,a_*\Orb_{X}) = \RHom_{\Orb_X}(\LDERF
    a^*P,\Orb_X) = \RDERF \Gamma(X,\LDERF a^*P^{\vee}) \in \DCAT^b_{\COH}(A).
  \]
  Hence, Theorem \ref{T:finiteness} implies that
  $a_*\Orb_X \in \COH(X_0)$. That is, $a$ is finite and finitely
  presented. By composition, $X \to \spec A$ is proper and of finite
  presentation.
\end{proof}

\section{Adjoints}\label{S:adjoints}
Throughout we let $X$ be an algebraic space. Consider a morphism of
ringed topoi $c \colon \mathcal{X} \to X_{\et}$. There is an adjoint
pair on the level of unbounded derived categories
\[
  \LDERF c^* \colon \DCAT(X) \rightleftarrows \DCAT(\mathcal{X}) \colon \RDERF c_*. 
\]
The inclusion $\DQCOH(X) \subseteq \DCAT(X)$ is fully faithful and also admits a right adjoint, the quasi-coherator $\RDERF Q_X \colon \DCAT(X) \to \DQCOH(X)$. It follows immediately that
\begin{enumerate}
\item the restriction of $\LDERF c^*$ to $\DQCOH(X)$ is left adjoint to
  $\RDERF Q_X\RDERF c_*$; and 
\item if $M\in \DQCOH(X)$, then the natural map $M \to \RDERF Q_X(M)$ is a
  quasi-isomorphism.
\end{enumerate}
We will let
\[
  \LDERF c^* \colon \DQCOH(X) \rightleftarrows
  \DCAT(\mathcal{X}) \colon \RDERF c_{\qcsubscript,*}
\]
denote the resulting adjoint pair. Let $M\in \DQCOH(X)$ and $\mathcal{N} \in \DCAT(\mathcal{X})$. Let
\[
\eta_M \colon M \to \RDERF c_{\qcsubscript,*}\LDERF c^*M \quad \mbox{and} \quad \epsilon_{\mathcal{N}} \colon \LDERF c^*\RDERF c_{\qcsubscript,*}\mathcal{N} \to \mathcal{N}
\]
denote the morphisms resulting from the adjunctions. 

We now use Theorem \ref{T:finiteness} to show that
$\RDERF c_{\qcsubscript,*}$ frequently preserves pseudo-coherence.
\begin{proposition}\label{P:adjoints}
  Let $A$ be a ring. Let $X \to \spec A$ be a quasi-compact,
  separated, and pseudo-coherent morphism of algebraic spaces. Let
  $c\colon \mathcal{X} \to X_{\et}$ be a morphism of ringed topoi. Let
  $*\in \{b,-\}$. If $\RDERF \Gamma(\mathcal{X},-)$ sends
  $\DPCOH{*}(\mathcal{X})$ to $\DPCOH{*}(A)$, then the
  restriction of $\RDERF c_{\qcsubscript,*}$ to
  $\DPCOH{*}(\mathcal{X})$ factors through $\DPCOH{*}(X)$.
\end{proposition}
\begin{proof}
  Let $\mathcal{M} \in \DPCOH{*}(\mathcal{X})$ and let $P\in \DQCOH(X)$ be perfect. Then $\LDERF c^*P^\vee \in \DQCOH(\mathcal{X})$ is perfect, so $\mathcal{M} \tensor^{\LDERF}_{\Orb_{\mathcal{X}}} \LDERF c^*P \in \DPCOH{*}(\mathcal{X})$. Hence,
  \begin{align*}
    \RHom_{\Orb_X}(P,\RDERF c_{\qcsubscript,*}\mathcal{M}) &\simeq \RHom_{\Orb_X}(\LDERF c^*P,\mathcal{M})\simeq \RDERF \Gamma(\mathcal{X},\LDERF c^*P^\vee \tensor^{\LDERF}_{\Orb_{\mathcal{X}}} \mathcal{M}) \in \DPCOH{*}(A).
  \end{align*}
  By Theorem \ref{T:finiteness},
  $\RDERF c_{\qcsubscript,*}\mathcal{M} \in \DPCOH{*}(X)$.
\end{proof}
\begin{remark}
A variant of
Proposition \ref{P:adjoints} that is valid for finite type cohomological
functors for proper schemes over noetherian bases, which generalizes \cite[Thm.~1.1]{MR1996800}, appears in \cite{Neeman_Approx}.
\end{remark}
We return to our general discussion. The categories $\DQCOH(X)$ and $\DCAT(\mathcal{X})$ are symmetric
monoidal and the derived pullback $\LDERF c^*$ is strong
monoidal. This lets us apply the formalism in Appendix
\ref{A:projection-formula} to our situation. We record some
consequences here.
\begin{lemma}\label{L:projection_formula}
  If $M \in \DQCOH(X)$ and $\mathcal{N} \in \DCAT(\mathcal{X})$, then
 there is a natural projection morphism
  \[
    \pi_{M,\mathcal{N}} \colon  M \tensor^{\LDERF}_{\Orb_X} (\RDERF c_{\qcsubscript,*}\mathcal{N})
    \to \RDERF c_{\qcsubscript,*}(\LDERF c^* M
    \otimes_{\Orb_{\mathcal{X}}}^{\LDERF} \mathcal{N}).
  \]
  This is an isomorphism if $M$ is perfect or $\Orb_{\mathcal{X}}$
  is a compact object of $\DCAT(\mathcal{X})$.
\end{lemma}
\begin{proof}
  If $M$ is perfect, then it is a dualizable object of $\DQCOH(X)$
  \cite[Lem.~4.3]{perfect_complexes_stacks}. Hence, the projection
  morphism is an isomorphism in this case by Theorem
  \ref{AT:projection-formula}.  If $\Orb_{\mathcal{X}}$ is a compact
  object of $\DCAT(\mathcal{X})$, then
  $\LDERF c^* \colon \DQCOH(X) \to \DCAT(\mathcal{X})$ preserves
  compact objects \cite[Lem.~2.3(2)]{telescope-stacks}, so its right
  adjoint $\RDERF c_{\qcsubscript,*}$ preserves small coproducts
  \cite[Thm.~5.1]{MR1308405}. Hence, the full subcategory of
  $\DQCOH(X)$ consisting of those $M$ for which $\pi_{M,\mathcal{N}}$
  is an isomorphism is localizing and contains the perfect
  complexes. By Thomason's localization theorem
  \cite[Thm.~2.1.2]{MR1308405}, the result follows.
\end{proof}
\begin{remark}\label{R:compact-mod}
  The condition that $\Orb_{\mathcal{X}}$ is a compact object of
  $\DCAT(\mathcal{X})$ is subtle, but frequently
  satisfied. A useful criterion is \cite[Tag
  \spref{094D}]{stacks-project}, which shows that it is sufficient for
  cohomology of abelian sheaves on $\mathcal{X}$ to commute with filtered
  colimits and have finite cohomological dimension. It follows that
  $\Orb_{\mathcal{X}}$ is a compact object of $\DCAT(\mathcal{X})$
  whenever $\mathcal{X}$ is equivalent to the topos of 
  \begin{enumerate}
  \item a noetherian space of finite Krull
    dimension \cite[Thm.~3.6.5]{MR0102537}; or
  \item a spectral space of finite Krull dimension \cite[Thm.~4.5]{MR1179103}; or
  \item a compact Hausdorff space of finite cohomological dimension.
  \end{enumerate}
\end{remark}
\section{Equivalences}\label{S:equivalences}
In this section, the following setup will feature frequently. 
\begin{setup}\label{set:Zequiv}
  Let $i \colon Z \to X$ be a morphism of quasi-compact and
  quasi-separated algebraic spaces. Consider a $2$-commutative
  diagram of ringed topoi:
  \[
    \xymatrix{\mathcal{Z} \ar[r]^{c_Z} \ar[d]_{i'} &
      \ar[d]^{i} Z_{\et}\\ \mathcal{X} \ar[r]^c & X_{\et}. }
  \]
  We will also assume that
  \begin{enumerate}[label=\thetheorem.\arabic*]
  \item\label{seti:Zequiv:equiv}
    $\LDERF c_Z^* \colon \DPCOH{-}(Z) \to \DPCOH{-}(\mathcal{Z})$ is
    an equivalence of categories; and
  \item\label{seti:Zequiv:proj}
    $\pi_{N,\RDERF i'_*\mathcal{Q}} \colon {N}
    \otimes^{\LDERF}_{\Orb_{{X}}} \RDERF c_{\qcsubscript,*}
   \RDERF i'_*\mathcal{Q} \simeq \RDERF c_{\qcsubscript,} (\LDERF c^*{N}
    \otimes^{\LDERF}_{\Orb_{\mathcal{X}}}\RDERF i'_*\mathcal{Q})$
    for all ${N} \in \DPCOH{-}({X})$ and
    $\mathcal{Q} \in \DPCOH{-}(\mathcal{Z})$.
  \end{enumerate}
  Setup \ref{set:Zequiv} is said to be \emph{tor-independent} if the
  following condition holds:
  \begin{enumerate}[label=\thetheorem.3]
  \item \label{seti:Zequiv:tori} $\LDERF c^*\RDERF i_* Q \simeq \RDERF i'_* \LDERF c_Z^*Q$ for
    all $Q\in \DPCOH{-}(Z)$.
\end{enumerate}
\end{setup}
\begin{example}\label{E:Zequiv}
  If $c$ is the formal completion along a closed immersion
  $i\colon Z \hookrightarrow X$ or when $Z$ and $\mathcal{Z}$ are
  points with the same residue field, then \ref{seti:Zequiv:equiv}
  holds. Indeed, here we simply take $\mathcal{Z} = Z_{\et}$. Lemma
  \ref{L:Zequiv:projection} will show that \ref{seti:Zequiv:proj} also
  holds in this setting. Further, if $c$ and $i$ are tor-independent
  (e.g., $c$ is flat or $X$ is noetherian and $c$ is the formal
  completion along $i$), then Lemma \ref{L:Zequiv-tor-ind-qaff} implies that
  \ref{seti:Zequiv:tori} holds.
\end{example}
The next result is a sanity check for \ref{seti:Zequiv:equiv}. 
\begin{lemma}\label{L:Zequiv-restricts}
  Assume \ref{seti:Zequiv:equiv}. Then the restriction of
  $\RDERF c_{Z,\qcsubscript,*}$ to
  $\DPCOH{-}(\mathcal{Z})$ induces an equivalence:
  \[
    \LDERF c^*_Z \colon \DPCOH{-}(Z) \leftrightarrows
    \DPCOH{-}(\mathcal{Z}) \colon \RDERF
    c_{Z,\qcsubscript,*}.
  \]
  \end{lemma}
  \begin{proof}
    Let $P$, $Q \in \DPCOH{-}(Z)$. Then there are isomorphisms:
  \[
    \RHom_{\Orb_{Z}}(P,Q) \simeq
    \RHom_{\Orb_{\mathcal{Z}}}(\LDERF c_Z^*P,\LDERF
    c_Z^*Q) \simeq \RHom_{\Orb_{Z}}(P,\RDERF
    c_{Z,\qcsubscript,*} \LDERF c_Z^*Q).
  \]
  Let $H_Q$ be the cone of the adjunction morphism
  $\eta_{Z,Q} \colon Q \to \RDERF c_{Z,\qcsubscript,*}\LDERF c_Z^*Q$.
  Then $\RHom_{\Orb_{Z}}(P,H_Q) \simeq 0$ for all
  $P \in \DPCOH{-}(Z)$. But $\DPCOH{-}(Z_\lambda)$ contains the
  perfect complexes of $Z$, so $H_Q \simeq 0$
  \cite[Thm.~A]{perfect_complexes_stacks}. That is, $\eta_{Z,Q}$ is an
  isomorphism for all $Q\in \DPCOH{-}(Z)$. Now let
  $\mathcal{Q} \in \DPCOH{-}(\mathcal{Z})$. Then $Q\in \DPCOH{-}(Z)$
  and there is an isomorphism $\mathcal{Q} \simeq \LDERF c_Z^*Q$. By
  what we have proved so far, it follows that
  $\RDERF c_{Z,\qcsubscript,*}\mathcal{Q} \simeq \RDERF
  c_{Z,\qcsubscript,*}\LDERF c_Z^*Q \simeq Q$. That is,
  $\RDERF c_{Z,\qcsubscript,*}$ restricts to a functor from
  $\DPCOH{-}(\mathcal{Z})$ to $\DPCOH{-}(Z)$. It
  follows immediately from general nonsense that
  $\RDERF c_{Z,\qcsubscript,*}$ is right adjoint to
  $\LDERF c_Z^* \colon \DPCOH{-}(Z) \to
  \DPCOH{-}(\mathcal{Z})$ and we have the claimed adjoint
  equivalence.
\end{proof}
The following lemma shows that \ref{seti:Zequiv:proj} is a very mild
condition.
\begin{lemma}\label{L:Zequiv:projection}
  Assume \ref{seti:Zequiv:equiv}. If one of the following hold, then
  \ref{seti:Zequiv:proj} holds.
  \begin{enumerate}
  \item\label{LI:Zequiv:projection:nice} The projection morphism
    $\pi_{\mathcal{N},\mathcal{Q}}^{\LDERF i'^*,\RDERF i'_*}$ is an
    isomorphism for all $\mathcal{N} \in \DPCOH{-}(\mathcal{X})$ and
    $\mathcal{Q} \in \DPCOH{-}(\mathcal{Z})$.
  \item\label{LI:Zequiv:projection:pedantic} The projection morphism
    $\pi_{\LDERF c^*N,\mathcal{Q}}^{\LDERF i'^*,\RDERF i'_*}$ is an
    isomorphism for all $N \in \DPCOH{-}(X)$ and
    $\mathcal{Q} \in \DPCOH{-}(\mathcal{Z})$.
  \item\label{LI:Zequiv:projection:useful} There exists an integer $K$
    such that for each $\mathcal{U}$ of $\mathcal{X}$,
    $\trunc{>K}\RDERF i'_{\mathcal{U},*}\mathcal{Q} \simeq 0$ for all
    $\mathcal{Q} \in \DPCOH{\leq 0}(i'^{-1}\mathcal{U})$, where
    $i'_{\mathcal{U}} \colon i'^{-1}(\mathcal{U}) \to \mathcal{U}$ is
    the induced morphism.
  \end{enumerate}
\end{lemma}
\begin{proof}
  Certainly, we have
  \itemref{LI:Zequiv:projection:nice}$\Rightarrow$\itemref{LI:Zequiv:projection:pedantic}. We
  next show
  \itemref{LI:Zequiv:projection:useful}$\Rightarrow$\itemref{LI:Zequiv:projection:nice}:
  let $\mathcal{N} \in \DPCOH{\leq a}(\mathcal{X})$ and
  $\mathcal{Q} \in \DPCOH{\leq b}(\mathcal{Z})$ for some integer
  $b$. Let $n\in \Z$; then it remains to prove that the projection morphism
  \[
    \pi_{\mathcal{N},\mathcal{Q}}^{\LDERF i'^*,\RDERF i'_*} \colon \mathcal{N} \otimes^{\LDERF}_{\Orb_{\mathcal{X}}} \RDERF i'_*\mathcal{Q} \to \RDERF i'_*(\LDERF i'^*\mathcal{N} \otimes^{\LDERF}_{\Orb_{\mathcal{Z}}} \mathcal{Q})
  \]
  is $n$-coconnected. This assertion is local on $\mathcal{X}$. Thus,
  by pseudo-coherence of $\mathcal{N}$, we may assume that there is an
  $(n-K-b)$-coconnected morphism
  $\phi \colon \mathcal{P} \to \mathcal{N}$, where $\mathcal{P}$ is
  perfect. Equivalently, if $\mathcal{E}$ sits in a distinguished
  triangle:
  \[
    \xymatrix{\mathcal{P} \ar[r]^\phi & \mathcal{N} \ar[r] & \mathcal{E} \ar[r] & \mathcal{P}[1]},
  \]
  then $\trunc{>n-K-b}\mathcal{E}\simeq 0$. Now $\mathcal{P}$ is
  perfect, so
  $\pi_{\mathcal{P},\mathcal{Q}}^{\LDERF i'^*,\RDERF i'_*}$ is an
  isomorphism (Theorem \ref{AT:projection-formula}). However, we
  certainly have $\trunc{>K+b}\RDERF i'_*\mathcal{Q} \simeq 0$, so
  $ \trunc{>n}(\mathcal{E} \otimes^{\LDERF}_{\Orb_{\mathcal{X}}}
  \RDERF i'_*\mathcal{Q}) \simeq 0$. Similarly,
  $\trunc{>n-K}(\LDERF i'^*\mathcal{E}
  \otimes^{\LDERF}_{\Orb_{\mathcal{Z}}} \mathcal{Q}) \simeq 0$ and so
  $\trunc{>n}\RDERF i'_*(\LDERF i'^*\mathcal{E}
  \otimes^{\LDERF}_{\Orb_{\mathcal{Z}}} \mathcal{Q}) \simeq 0$. It
  follows immediately that the projection morphism is $n$-coconnected.

  It remains to prove that \itemref{LI:Zequiv:projection:pedantic}
  implies \ref{seti:Zequiv:proj}. This is essentially just the
  functoriality properties of the projection formula together with
  Lemma \ref{L:Zequiv-restricts}. Specifically, we apply Lemma
  \ref{L:projection-composition} and Remark
  \ref{R:projection-natural-iso} with $\mathcal{C}=\DQCOH(X)$,
  $\mathcal{D} = \DCAT(\mathcal{X})$, $\mathcal{C}' = \DQCOH(Z)$, and
  $\mathcal{D}'=\DCAT(\mathcal{Z})$ with the natural functors and
  adjoints already described. In more details: there is a commutative diagram
  \[
    \xymatrix{ \ar[d] N \otimes^{\LDERF}_{\Orb_X} \RDERF
      c_{\qcsubscript,*}\RDERF i'_*\mathcal{Q} \ar[r] & \RDERF
      c_{\qcsubscript,*}(\LDERF c^*N
      \otimes^{\LDERF}_{\Orb_{\mathcal{X}}} \RDERF i'_*\mathcal{Q})
      \ar[d] \\ N \otimes^{\LDERF}_{\Orb_X} \RDERF i_* \RDERF
      c_{Z,\qcsubscript,*}\mathcal{Q} \ar[d]  &  \RDERF
      c_{\qcsubscript,*} \RDERF i'_*(\LDERF i'^*N
      \otimes^{\LDERF}_{\Orb_{\mathcal{Z}}} \mathcal{Q}) \ar[d] \\
      \RDERF i_*(\LDERF i^*N \otimes^{\LDERF}_{\Orb_Z} \RDERF
      c_{Z,\qcsubscript,*}\mathcal{Q}) \ar[r] & \RDERF i_*\RDERF
      c_{Z,\qcsubscript,*}(\LDERF i'^*N
      \otimes^{\LDERF}_{\Orb_{\mathcal{Z}}} \mathcal{Q}). }
  \]
  The map along the top is what we want to show is an isomorphism. The
  upper morphism on the right is an isomorphism by
  \itemref{LI:Zequiv:projection:pedantic}. The lower morphism on the
  right and the upper morphism on the left are isomorphisms by
  functoriality. The lower morphism on the left is an
  isomorphism by the usual projection formula for quasi-compact and
  quasi-separated morphisms of algebraic spaces
  \cite[Cor.~4.12]{perfect_complexes_stacks}. Finally, the bottom
  morphism is an isomorphism by Remark \ref{R:projection-equivalence}
  applied to the equivalence of Lemma \ref{L:Zequiv-restricts}.
\end{proof}
Tor-independence is more subtle to arrange, however.
\begin{lemma}\label{L:Zequiv-tor-ind-qaff}
  Assume \ref{seti:Zequiv:equiv}. If
  \begin{enumerate}
  \item $c$ and $i$ are tor-independent morphisms;
  \item $i$ factors as
    $Z \xrightarrow{\tilde{\imath}} \tilde{X} \xrightarrow{j} X$,
    where $\tilde{\imath}$ is affine and $j$ is \'etale; and
  \item
    $\DPCOH{-}(\mathcal{Z}) \simeq
    \DPCOH{-}(c^{-1}(\tilde{X}),\tilde{\imath}'_*\Orb_{\mathcal{Z}})$,
  \end{enumerate}
  then \ref{seti:Zequiv:tori} holds. 
\end{lemma}
\begin{proof}
  Let
  $\mathcal{Z} \xrightarrow{\tilde{\imath}'} c^{-1}(\tilde{X})
  \xrightarrow{j'} \mathcal{X}$ and
  $\tilde{c} \colon c^{-1}(\tilde{X}) \to \tilde{X}_{\et}$ be the induced
  morphisms.  We may assume that $j$ is a cover, so $j'$ is a
  cover. Let $Q \in \DPCOH{-}(Z)$; then we must prove that the base
  change morphism
  $\LDERF c^*\RDERF i_*Q \to \RDERF i'_* \LDERF c_Z^*Q$ is an
  isomorphism. Since $j'$ is covering, it suffices to prove that
  $\LDERF j'^*\LDERF c^*\RDERF i_*Q \to \LDERF j'^*\RDERF i'_* \LDERF
  c_Z^*Q$ is an isomorphism. Since the following diagram commutes:
  \[
    \xymatrix{\LDERF j'^*\LDERF c^* \RDERF i_*Q \ar[r] \ar[d] & \LDERF j'^* \RDERF i'_*\LDERF c_Z^*Q \ar[d]\\
      \ar[d] \LDERF \tilde{c}^* \LDERF j^*  \RDERF j_* \RDERF \tilde{\imath}_* Q & \LDERF j'^* \RDERF j'_* \RDERF \tilde{\imath}'_* \LDERF c_Z^* Q \ar[d]\\
      \LDERF \tilde{c}^* \RDERF \tilde{\imath}_*Q \ar[r] & \RDERF
      \tilde{\imath}'\LDERF c_Z^*Q,}
  \]
  where the vertical morphisms are all isomorphisms, and the bottom
  morphism is the base change map, we are reduced to the situation
  where $X=\tilde{X}$. The result now follows from Lemma
  \ref{L:tor-ind-topoi}.
\end{proof}

The following lemma collects
some of the technicalities of this section.
\begin{lemma}\label{L:Zequiv-locus}
Assume Setup \ref{set:Zequiv}. Let $N\in \DPCOH{-}(X)$ and define
\begin{align*}
      \mathcal{V}_{N} &= \{ \mathcal{Q} \in \DCAT(\mathcal{X}) \suchthat \mbox{$\pi_{N,\mathcal{Q}}$ is an isomorphism}\},\\
    \mathcal{T} &= \{ \mathcal{Q} \in \DCAT(\mathcal{X}) \suchthat \mbox{$\epsilon_{\mathcal{Q}}$ is an isomorphism}\}, \,\mbox{and} \\
    \mathcal{S} &= \{ P \in \DQCOH(X) \suchthat \mbox{$\eta_P$ is an isomorphism}\}.
  \end{align*}
  Then
  $\langle \RDERF i'_*\DPCOH{-}(\mathcal{Z})\rangle \subseteq
  \mathcal{V}_N$. If tor-independent, then
  $\langle \RDERF i'_*\DPCOH{-}(\mathcal{Z})\rangle \subset
  \mathcal{T}$ and
  $\langle \RDERF i_*\DPCOH{-}(Z)\rangle \subseteq \mathcal{S}$.
\end{lemma}
\begin{proof}
  We may view $\mathcal{V}_N$, $\mathcal{T}$, and $\mathcal{S}$ as
  full subcategories of $\DCAT(\mathcal{X})$, $\DCAT(\mathcal{X})$,
  and $\DQCOH(X)$, respectively. They are obviously triangulated and
  thick subcategories.  It remains to prove the following.
  \begin{enumerate}
  \item If $\mathcal{Q}_0 \in \DPCOH{-}(\mathcal{Z})$, then
    $\RDERF i'_*\mathcal{Q}_0 \in \mathcal{V}_N$.  This is \ref{seti:Zequiv:proj}.
  \item If tor-independent and
    $\mathcal{Q}_0 \in \DPCOH{-}(\mathcal{Z})$, then
    $\RDERF i'_*\mathcal{Q}_0 \in \mathcal{T}$. To see this: the diagram
    \[
      \xymatrix{\LDERF c^*\RDERF c_{\qcsubscript,*}\RDERF i'_*\mathcal{Q}_0 \ar[r]\ar[d]_{\epsilon_{\RDERF i'_*\mathcal{Q}_0}} & \LDERF c^*\RDERF i_*\RDERF c_{Z,\qcsubscript,*}\mathcal{Q}_0 \ar[d] \\ \RDERF i'_*\mathcal{Q}_0 & \ar[l]_-{\RDERF i'_*\epsilon_{\mathcal{Q}_0}} \RDERF i'_*\LDERF c_Z^*\RDERF c_{Z,\qcsubscript,*}\mathcal{Q}_0}
    \]
    commutes. The claim now follows from functoriality (the top morphism) Lemmas
    \ref{L:Zequiv-restricts} (the bottom morphism) and
    \ref{seti:Zequiv:tori} (the right morphism).
  \item If tor-independent and ${Q}_0 \in \DPCOH{-}({Z})$, then
    $\RDERF i_*{Q}_0 \in \mathcal{S}$: this similar to the previous step, so is omitted. \qedhere
  \end{enumerate}
\end{proof}

  We now introduce a key definition. We appreciate that it is
  difficult to parse. Such a definition appears necessary, however, to
  treat the lack of tor-independence that appears in the
  non-noetherian situation as well as the subtleness of the projection
  morphism. When tor-independence is available, Proposition
  \ref{P:Zequi-torind} provides a useful criterion for
  $Z$-equivalence.
\begin{definition}\label{D:Zequiv}
  Assume Setup \ref{set:Zequiv}. Let
  $\mathcal{M} \in \DCAT(\mathcal{X})$, $\mathcal{N}\in \DPCOH{-}(\mathcal{X})$. We say that $c$ is
  \begin{enumerate}[label=(\alph*), ref=\alph*, series=Zequiv]
  \item \label{Zequiv:faithful} \emph{faithful} along $\mathcal{M}$ at $\mathcal{N}$ if 
    \[
      \nu_{\mathcal{M},\mathcal{N}} \colon \RDERF
      c_{\qcsubscript,*}\mathcal{M} \otimes^{\LDERF}_{\Orb_X} \RDERF
      c_{\qcsubscript,*}\mathcal{N} \to \RDERF c_{\qcsubscript,*}(\mathcal{M}
      \otimes^{\LDERF}_{\Orb_{\mathcal{X}}} \mathcal{N})
    \]
    is an isomorphism; and
  \item \label{Zequiv:equiv} an \emph{equivalence} along $\mathcal{M}$ at $\mathcal{N}$ if \itemref{Zequiv:faithful} holds and 
    \[
      {\mathcal{M}} \tensor \epsilon_{\mathcal{N}} \colon
      \mathcal{M} \tensor^{\LDERF}_{\Orb_{\mathcal{X}}} \LDERF
      c^*\RDERF c_{\qcsubscript,*}\mathcal{N} \to \mathcal{M}
      \otimes^{\LDERF}_{\Orb_{\mathcal{X}}}\mathcal{N}
    \]
    is an isomorphism.
  \end{enumerate}
  If these hold for all $\mathcal{N}$, then we omit the ``at
  $\mathcal{N}$''. If $\mathcal{M}=i'_*\Orb_{\mathcal{Z}}$, then we
  will replace ``$\mathcal{M}$'' with ``$Z$''.
\end{definition}
The simplest method to produce the above is to use the following.
\begin{proposition}\label{P:Zequi-torind}
  Assume Setup \ref{set:Zequiv}. Let $M \in \DQCOH(X)$ and
  $\mathcal{N} \in \DPCOH{-}(\mathcal{X})$.
  \begin{enumerate}
  \item\label{PI:Zequi-torind:eta} If $\eta_M$ is an isomorphism, then $c$ is faithful along $\LDERF c^*M$ at $\mathcal{N}$ if and only if $\pi_{M,\mathcal{N}}$ is an isomorphism.
  \item\label{PI:Zequi-torind:ti} If tor-independent and
    $M \in \langle \RDERF i_*\DPCOH{-}(Z)\rangle$, then $c$ is an
    equivalence along $\mathcal{M}$ at $\mathcal{N}$ if and only if
    $\pi_{M,\mathcal{N}}$ is an isomorphism.
\end{enumerate}
\end{proposition}
\begin{proof}
  By definition of the projection morphism \eqref{eq:projection}, the
  following diagram commutes:
  \[
    \xymatrix@C+3pc{ M \tensor^{\LDERF}_{\Orb_X}(\RDERF
      c_{\qcsubscript,*}\mathcal{N}) \ar[r]^-{\eta_M \tensor {\RDERF c_{\qcsubscript,*}\mathcal{N}}} \ar[dr]_-{\pi_{M,\mathcal{N}}} & \RDERF
      c_{\qcsubscript,*}\LDERF c^*M \tensor^{\LDERF}_{\Orb_X} (\RDERF
      c_{\qcsubscript,*}\mathcal{N}) \ar[d]^{\nu_{\LDERF c^*M, \mathcal{N}}} \\ & \RDERF
      c_{\qcsubscript,*}(\LDERF c^*M
      \tensor^{\LDERF}_{\Orb_{\mathcal{X}}} \mathcal{N}).}
  \]
  This proves \itemref{PI:Zequi-torind:eta}. If $M \in \langle \RDERF i_*\DPCOH{-}(Z)\rangle$ and $c$ and $i$
  are tor-independent, then 
  $\eta_M$ is an
  isomorphism (Lemma \ref{L:Zequiv-locus}). By Lemma
  \ref{L:projections-equal}, the following diagram commutes:
  \[
    \xymatrix{\LDERF c^*(M \tensor^{\LDERF}_{\Orb_{{X}}} \RDERF
      c_{\qcsubscript,*}\mathcal{N}) \ar[r] \ar[d]_{\LDERF
        c^*\pi_{M,\mathcal{N}}} & \LDERF c^*M
      \tensor^{\LDERF}_{\Orb_{\mathcal{X}}} \LDERF c^*\RDERF
      c_{\qcsubscript,*}\mathcal{N} \ar[d]^{\LDERF c^*M \tensor
        \epsilon_{\mathcal{N}}} \\ \LDERF c^*\RDERF
      c_{\qcsubscript,*}(\LDERF c^*M
      \tensor^{\LDERF}_{\Orb_{\mathcal{X}}} \mathcal{N})
      \ar[r]^-{\epsilon_{\LDERF c^*M
          \tensor_{\Orb_{\mathcal{X}}}^{\LDERF} \mathcal{N}}} & \LDERF
      c^*M \tensor^{\LDERF}_{\Orb_{\mathcal{X}}} \mathcal{N}.}
  \]
  The top morphism is an isomorphism, as is the bottom (Lemmas
  \ref{L:Zequiv-locus} and \ref{L:projection-topoi-affine}). The
  stated equivalence follows.
\end{proof}
Many examples are provided by the following two results.
\begin{corollary}\label{C:Zequiv-examples-tor-ind}
  Assume tor-independent Setup \ref{set:Zequiv}. If
  \begin{enumerate}
  \item\label{LI:Zequiv-examples-tor-ind:cartier} $i$ is a Cartier divisor; or 
  \item\label{LI:Zequiv-examples-tor-ind:perfect} $i_*\Orb_Z$ is perfect; or 
  \item\label{LI:Zequiv-examples-tor-ind:compact} $\Orb_{\mathcal{X}}$
    is a compact object of $\DCAT(\mathcal{X})$ (see Remark \ref{R:compact-mod});
  \end{enumerate}
  then $c$ is an equivalence along $Z$.
\end{corollary}
\begin{proof}
  We use the criterion of Proposition \ref{P:Zequi-torind}. Case
  \itemref{LI:Zequiv-examples-tor-ind:cartier} is a special case of
  \itemref{LI:Zequiv-examples-tor-ind:perfect}. In cases
  \itemref{LI:Zequiv-examples-tor-ind:perfect} and
  \itemref{LI:Zequiv-examples-tor-ind:compact} the projection
  morphism is an isomorphism by Lemma \ref{L:projection_formula}. 
\end{proof}
\begin{corollary}\label{C:Zequiv-examples-tor-ind-coherent}
  Assume tor-independent Setup \ref{set:Zequiv} and
  $i$ is a closed immersion. If $\Orb_X$ is coherent, then there is a
  perfect complex $M \in \langle \RDERF i_*\DPCOH{-}(Z) \rangle$ with
  $\trunc{\geq 0}M \simeq \Orb_Z$ such that $c$ is an equivalence along
  $\LDERF c^*M$.
\end{corollary}
\begin{proof}
  By perfect approximation \cite[Tag \spref{08HP}]{stacks-project},
  there exists a perfect complex $M \in \DCAT^{\leq 0}_{\COH,|Z|}(X)$
  with $\trunc{\geq 0}M \simeq \Orb_Z$. By Lemma \ref{L:good-perfect},
  $M \in \thick{\RDERF i_{*}\DCAT^b_{\COH}(Z)}$, so
  $\LDERF c^*M \in \thick{\RDERF i'_*\DPCOH{-}(\mathcal{Z})}$ (Lemma
  \ref{L:tor-ind-topoi}). Now apply Lemma \ref{L:projection_formula}
  and Proposition \ref{P:Zequi-torind}.
\end{proof}

Lacking tor-independence, these notions can be quite subtle. We will give some interesting examples at the end of this section, however. In the meantime, we content ourselves with the following useful lemma.
\begin{lemma}\label{L:Zequiv<=Zfaithful}
  Assume Setup \ref{set:Zequiv}. Let
  $\mathcal{M}_0 \in \DPCOH{-}(\mathcal{Z})$,
  $\mathcal{N} \in \DPCOH{-}(\mathcal{X})$. If ${\RDERF c_{\qcsubscript,*}\mathcal{N} \in
    \DPCOH{-}({X})}$ 
and $c$ is faithful along $\RDERF i'_*\mathcal{M}_0$ at
$\mathcal{N}$, then $c$ is an equivalence along
$\RDERF i'_*\mathcal{M}_0$ at $\mathcal{N}$.
\end{lemma}
\begin{proof}
  We have the following commutative diagram:
  \[
    \xymatrix@C+3.5pc{\RDERF c_{\qcsubscript,*}( \LDERF c^*\RDERF
      c_{\qcsubscript,*}\mathcal{N}
      \tensor^{\LDERF}_{\Orb_{\mathcal{X}}} \RDERF i'_*\mathcal{M}_0)
      \ar[r]^-{\pi_{\RDERF c_{\qcsubscript,*}\mathcal{N},\RDERF
          i'_*\mathcal{M}_0}} \ar[dr]_-{\RDERF c_{\qcsubscript,*}(\epsilon_{\mathcal{N}} \tensor \RDERF i'_*\mathcal{M}_0)} & \RDERF
      c_{\qcsubscript,*}(\mathcal{N}) \tensor^{\LDERF}_{\Orb_{X}}
      \RDERF c_{\qcsubscript,*}\RDERF i'_*\mathcal{M}_0 \ar[d]^-{\nu_{\mathcal{N},\RDERF i'_*\mathcal{M}_0}}\\ &
      \RDERF c_{\qcsubscript,*}(\mathcal{N}
      \tensor^{\LDERF}_{\Orb_{\mathcal{X}}} \RDERF i'_*\mathcal{M}_0).
    }
  \]
  Since $c$ is faithful along $\RDERF i'_*\mathcal{M}_0$ at
  $\mathcal{N}$, the vertical map is an isomorphism. By Lemma
  \ref{L:Zequiv-locus}, the horizontal map is an isomorphism. It
  follows that the diagonal map is an isomorphism. Let $\mathcal{Q}$
  be a cone for $\epsilon_{\mathcal{N}}$; then
  $\mathcal{Q} \in \DPCOH{-}(\mathcal{X})$. Hence,
  $\mathcal{Q} \tensor^{\LDERF}_{\Orb_{\mathcal{X}}}
  \RDERF i'_*\mathcal{M}_0 \simeq \RDERF i'_*(\LDERF i'^*\mathcal{Q} \tensor^{\LDERF}_{\Orb_{\mathcal{Z}}} \mathcal{M}_0)$
  (Lemma \ref{L:projection-topoi-affine}). Let $\mathcal{Q}_0=\LDERF i'^*\mathcal{Q} \tensor^{\LDERF}_{\Orb_{\mathcal{Z}}} \mathcal{M}_0$; then $\mathcal{Q}_0 \in \DPCOH{-}(\mathcal{Z})$ and 
  \[
    0 \simeq \RDERF c_{\qcsubscript,*}(\mathcal{Q}
    \tensor^{\LDERF}_{\Orb_{\mathcal{X}}} \RDERF i'_*\mathcal{M}_0)
    \simeq \RDERF c_{\qcsubscript,*}\RDERF i'_*\mathcal{Q}_0 \simeq
    \RDERF i'_*\RDERF c_{Z,\qcsubscript,*}\mathcal{Q}_0.
  \]
  It follows immediately that $\mathcal{Q}_0 \simeq 0$ and the claim
  follows.
\end{proof}
The whole reason for introducing these notions is the following key result.
\begin{proposition}\label{P:Zequi-faithful-use}
  Assume Setup \ref{set:Zequiv}. Let
  $\mathcal{M} \in \DCAT(\mathcal{X})$ and
  $\mathcal{N} = \LDERF c^*N$, where $N\in \DPCOH{-}(X)$. If
  $\pi_{N,\mathcal{M}}$ is an isomorphism, then the following conditions are
  equivalent:
  \begin{enumerate}
  \item $c$ is faithful along $\mathcal{M}$ at $\mathcal{N}$;
  \item ${\RDERF c_{\qcsubscript,*}\mathcal{M}}\tensor \eta_N$ is an isomorphism. 
  \end{enumerate}
\end{proposition}
\begin{proof}
  This is immediate from the commutativity of the following diagram \eqref{eq:projection}:
  \[
    \begin{gathered}[b]
    \xymatrix@C+3pc{ N \tensor^{\LDERF}_{\Orb_X}(\RDERF
      c_{\qcsubscript,*}\mathcal{M}) \ar[r]^-{\eta_N \tensor {\RDERF c_{\qcsubscript,*}\mathcal{M}}} \ar[dr]_-{\pi_{N,\mathcal{M}}} & \RDERF
      c_{\qcsubscript,*}\LDERF c^*N \tensor^{\LDERF}_{\Orb_X} (\RDERF
      c_{\qcsubscript,*}\mathcal{M}) \ar[d]^{\nu_{\LDERF c^*N, \mathcal{M}}} \\ & \RDERF
      c_{\qcsubscript,*}(\LDERF c^*N
      \tensor^{\LDERF}_{\Orb_{\mathcal{X}}} \mathcal{M}). }\\[-\dp\strutbox]
  \end{gathered}
  \qedhere
  \]
\end{proof}
We conclude this section with more methods to produce examples. 
\begin{lemma}\label{L:Zequiv_pullback}
  Assume Setup \ref{set:Zequiv}.
  \begin{enumerate}
  \item \label{LI:Zequiv_pullback:global} $c$ is faithful along $\mathcal{M}$ if and only if
    \[
      \RDERF \Gamma(X,\RDERF c_{\qcsubscript,*}\mathcal{M} \tensor^{\LDERF}_{\Orb_X} \RDERF c_{\qcsubscript,*}\mathcal{N}) \to \RDERF \Gamma(\mathcal{X},\mathcal{M} \tensor^{\LDERF}_{\Orb_{\mathcal{X}}} \mathcal{N})
    \]
    is an isomorphism for all $\mathcal{N}\in \DPCOH{-}(\mathcal{X})$.
  \item\label{LI:Zequiv_pullback:ideal} If $f \colon X \to \spec A$ is flat; 
    $Z=X  \otimes_A B$, where $B$ is an $A$-algebra; and 
    \[
      B\tensor^{\LDERF}_A \RDERF \Gamma(\mathcal{X},\mathcal{N}) \to \RDERF \Gamma(\mathcal{X},\LDERF i'^*\mathcal{N})
    \]
    is an isomorphism for all
    $\mathcal{N} \in \DPCOH{-}(\mathcal{X})$; then $c$ is faithful along $Z$.
\end{enumerate}
\end{lemma}
\begin{proof}
  For \itemref{LI:Zequiv_pullback:global}: the necessity is clear. For
  the sufficiency, the perfect complexes compactly generate
  $\DQCOH(X)$ \cite[Thm.~A]{perfect_complexes_stacks}, so it suffices
  to prove that
  \[
    \RHom_{\Orb_X}(P,\RDERF c_{\qcsubscript,*}\mathcal{M} \tensor^{\LDERF}_{\Orb_X} \RDERF c_{\qcsubscript,*}\mathcal{N}) \to \RHom_{\Orb_X}(P,\RDERF c_{\qcsubscript,*}(\mathcal{M} \tensor^{\LDERF}_{\Orb_{\mathcal{X}}} \mathcal{N}))
  \]
  is an isomorphism for all perfect $P$. Since perfects are
  dualizable, the morphism above is an isomorphism if and only if the
  following morphism is an isomorphism:
  \[
    \RDERF \Gamma(X,P^\vee \tensor^{\LDERF}_{\Orb_X}\RDERF c_{\qcsubscript,*}\mathcal{M} \tensor^{\LDERF}_{\Orb_X} \RDERF c_{\qcsubscript,*}\mathcal{N}) \to \RDERF \Gamma(X,P^\vee \tensor^{\LDERF}_{\Orb_X} \RDERF c_{\qcsubscript,*}(\mathcal{M} \tensor^{\LDERF}_{\Orb_{\mathcal{X}}} \mathcal{N})).
  \]
  The projection formula (Lemma \ref{L:projection_formula}) and adjunction says that this morphism is an isomorphism if and only if the following is an isomorphism:
  \[
    \RDERF \Gamma(X,\RDERF c_{\qcsubscript,*}(\LDERF c^*P^\vee
    \tensor^{\LDERF}_{\Orb_{\mathcal{X}}} \mathcal{M})
    \tensor^{\LDERF}_{\Orb_{{X}}}\RDERF c_{\qcsubscript,*} \mathcal{N})) \to \RDERF
    \Gamma(\mathcal{X},\LDERF c^*P^\vee
    \tensor^{\LDERF}_{\Orb_{\mathcal{X}}} \mathcal{M}
    \tensor^{\LDERF}_{\Orb_{\mathcal{X}}} \mathcal{N}).
  \]
  The claim follows.

  For \itemref{LI:Zequiv_pullback:ideal},
  we take
  $\mathcal{M} = c^*f^*(B)=i'_*\Orb_{\mathcal{Z}}$;
  then
  $\RDERF c_{\qcsubscript,*}\mathcal{M} \simeq \RDERF i_*\RDERF
  c_{Z,\qcsubscript,*}\Orb_{\mathcal{Z}} \simeq f^*(B) \simeq \LDERF f^*(B)$. We next observe that the usual projection formula \cite[Cor.~4.12]{perfect_complexes_stacks} implies that 
  \[
    \RDERF \Gamma(X,\LDERF f^*(B) \tensor^{\LDERF}_{\Orb_X} \RDERF
    c_{\qcsubscript,*}\mathcal{N}) \simeq B \tensor^{\LDERF}_{A}
    \RDERF \Gamma(X,\RDERF c_{\qcsubscript,*}\mathcal{N}) \simeq B
    \tensor^{\LDERF}_A \RDERF \Gamma(\mathcal{X},\mathcal{N}).
  \]
  The claim now follows from \itemref{LI:Zequiv_pullback:global}.
\end{proof}
\begin{remark}\label{R:quasi-affine-pullback}
  Lemma \ref{L:Zequiv_pullback} can easily be refined when $X$ is quasi-affine:
  \begin{enumerate}
  \item \label{RI:quasi-affine-pullback:global} $c$ is faithful along
    $\mathcal{M}$ at $\mathcal{N}$ if and only if the following is an isomorphism:
    \[
      \RDERF \Gamma(X,\RDERF c_{\qcsubscript,*}\mathcal{M}
      \tensor^{\LDERF}_{\Orb_X} \RDERF c_{\qcsubscript,*}\mathcal{N})
      \to \RDERF \Gamma(\mathcal{X},\mathcal{M}
      \tensor^{\LDERF}_{\Orb_{\mathcal{X}}} \mathcal{N}).
    \]
  \item\label{RI:quasi-affine-pullback:ideal} If $f \colon X \to \spec A$ is flat; 
    $B = A/I$, where $I \subseteq A$ is an ideal; and
    \[
      A/I\tensor^{\LDERF}_A \RDERF \Gamma(\mathcal{X},\mathcal{N}) \to \RDERF \Gamma(\mathcal{X},\LDERF i'^*\mathcal{N})
    \]
    is an isomorphism; then $c$ is a faithful along $Z$ at
    $\mathcal{N}$.
\end{enumerate}

\end{remark}
We have the following non-noetherian and non-tor-independent example
that comes from \cite[Tag \spref{0DIA}]{stacks-project}.
\begin{example}\label{E:Zequiv-non-noetherian}
  Let $\{A_n\}_{n\geq 0}$ be an inverse system of rings with
  surjective transition maps and locally nilpotent kernel. Let
  $A=\varprojlim_n A_n$. Let $X \to \spec A$ be a proper, flat and finitely presented 
  morphism of algebraic spaces. Let
  $I_n = \ker(A \to A_n)$ and let
  $\Orb_{X_n}=\Orb_X/I_n\Orb_{X_n}$. 
  Let $\mathcal{X}$ be the ringed topos with underlying
  topos $X_{\et}$ and sheaf of rings
  $\Orb_{\mathcal{X}}=\varprojlim_n \Orb_{X_n}$ in $\MOD(X)$. There is
  a morphism of ringed topoi $c\colon \mathcal{X} \to X_{\et}$
  corresponding to $\Orb_X \to \Orb_{\mathcal{X}}$. 
  Let
  $i_n \colon X_n \to X$ and $i_n' \colon X_n \to \mathcal{X}$ be the
  resulting morphisms; note that $c \circ i_n' = i_n$.
  
  We claim that if $\mathcal{M} \in \DPCOH{-}(\mathcal{X})$, then
  $\RDERF \Gamma(\mathcal{X},\mathcal{M}) \in \DPCOH{-}(A)$. For each
  $n\geq 0$ let
  $\mathcal{M}_n = \LDERF i_n'^*\mathcal{M} \in \DPCOH{-}(X_n)$. By
  \cite[Tag \spref{0CQF}]{stacks-project} and a local calculation,
  $\mathcal{M} \simeq \holim{n} \RDERF i_{n,*}'\mathcal{M}_n$ in
  $\DCAT(\mathcal{X})$.  Also $\RDERF \Gamma(\mathcal{X},-)$ preserves
  homotopy limits, so
  \begin{align*}
    M=\RDERF \Gamma(\mathcal{X},\mathcal{M}) &\simeq \holim{n} \RDERF \Gamma(\mathcal{X},\RDERF i_{n,*}'\mathcal{M}_n) \simeq \holim{n} \RDERF \Gamma(X_n,\mathcal{M}_n).
  \end{align*}
  Let $M_n = \RDERF \Gamma(X_n,\mathcal{M}_n)$. Then $M_n$ is a
  pseudo-coherent complex of $A_n$-modules (a special case of Kiehl's
  Finiteness Theorem, see \cite[Tag \spref{0CSD}]{stacks-project}) and
  the projection formula \cite[Cor.~4.12]{perfect_complexes_stacks}
  implies that:
  \[
    M_{n+1}\otimes^{\LDERF}_{A_{n+1}} A_n \simeq \RDERF
    \Gamma(X_{n+1},\mathcal{M}_{n+1} \otimes_{\Orb_{X_{n+1}}}^{\LDERF}
    \Orb_{X_n}) \simeq M_n.
  \]
  Thus, $M$ is $A$-pseudo-coherent and
  $\RDERF \Gamma(\mathcal{X},\mathcal{M}) \tensor^{\LDERF}_A A_n
  \simeq \RDERF \Gamma(X_n,\mathcal{M}_n)$ \cite[Tag
  \spref{0CQF}]{stacks-project}. By Proposition \ref{P:adjoints},
  $\RDERF c_{\qcsubscript,*}\mathcal{M} \in \DPCOH{-}(X)$.

  In Setup \ref{set:Zequiv}, we take $Z=X_0$. By the above and Lemma
  \ref{L:Zequiv_pullback}\itemref{LI:Zequiv_pullback:ideal}, $c$ is
  faithful along $Z$. By Lemma \ref{L:Zequiv<=Zfaithful}, $c$ is even
  an equivalence along $Z$.
\end{example}
\begin{example}\label{E:locally-principal}
  Let $X$ be a quasi-compact and quasi-separated algebraic space. Let
  $L$ be a line bundle on $X$ and let $s\in \Gamma(X,L)$. Let
  $s^\vee \colon L^\vee \to \Orb_X$ be the dual morphism and let
  $I=\im(s^\vee)$ and $K=\ker(s^\vee)$. Let
  $i \colon Z \hookrightarrow X$ be the closed immersion defined by
  $I$. That is, $i$ is the vanishing locus of $s$. Let
  $\Orb_{\mathcal{X}}$ be a sheaf of $\Orb_X$-algebras (not
  necessarily quasi-coherent) such that
  \begin{enumerate}[label=(\alph*), ref=\alph*]
  \item \label{EI:locally-principal:sub} $\Orb_X/I \to \Orb_{\mathcal{X}}/I\Orb_{\mathcal{X}}$ is
    an isomorphism; and
  \item \label{E:locally-principal:tor}
    $K \cong \ker(s^\vee  \tensor_{\Orb_X}
    \Orb_{\mathcal{X}})$.
  \end{enumerate}
  This holds, for example, when $X$ is locally noetherian, $s$ is a
  regular section of $L$ (i.e., $K=0$), and $\Orb_{\mathcal{X}}$ is
  the formal completion of $X$ along $I$. This also holds when $X$ is
  non-noetherian \cite[Tag \spref{0BNG}]{stacks-project}. More
  generally, it holds when $s$ is a regular section of $L$ and of
  $c^*L$. Note that if $s$ is a regular section, then $s$ remaining a
  regular section of $c^*L$ is easily seen to be equivalent to the
  tor-independence of $c$ and $i$.
  
  Let $Z=\mathcal{Z}$ in Setup \ref{set:Zequiv}. Let
  $\mathcal{X} = (X_{\et},\Orb_{\mathcal{X}})$ and
  $C=[L^\vee \xrightarrow{s^\vee} \Orb_X]$, which is perfect. We will
  establish the following:
  \begin{enumerate}
  \item \label{EI:locally-principle:eta} $\eta_C$ is an isomorphism;
  \item \label{EI:locally-principle:faith} $c$ is faithful along $\LDERF c^*C$;
    and
  \item \label{EI:locally-principle:equi} if
    $\mathcal{N} \in \DPCOH{-}(\mathcal{X})$ is such that
    $\mathcal{N} \tensor^{\LDERF}_{\Orb_X} C$ or
    $\mathcal{N} \tensor^{\LDERF}_{\Orb_X} K$ belongs to $\DQCOH(X)$,
    then $c$ is an equivalence along $\LDERF c^*C$ at $\mathcal{N}$.
  \end{enumerate}
  Condition \itemref{EI:locally-principle:equi} is of course trivially
  satisfied when $s$ is a regular section of $L$.

  We first prove \itemref{EI:locally-principle:eta}. Consider the morphism of distinguished triangles:
  \[
    \xymatrix{\COHO{-1}(C)[1] \ar[r] \ar[d] & C \ar[r] \ar[d] & \COHO{-0}(C)[0] \ar[r] \ar[d] & \COHO{-1}(C)[2] \ar[d]\\
    \COHO{-1}(\RDERF c_{\qcsubscript,}\LDERF c^*C)[1] \ar[r] & \RDERF c_{\qcsubscript,}\LDERF c^*C \ar[r] & \COHO{0}(\RDERF c_{\qcsubscript,}\LDERF c^*C)[0] \ar[r] & \COHO{-1}(\RDERF c_{\qcsubscript,}\LDERF c^*C)[2].}
  \]
  Now $\COHO{-0}(C) = \Orb_X/I$ and
  $\COHO{-0}(\LDERF c^*C) = \Orb_{\mathcal{X}}/I\Orb_{\mathcal{X}}$,
  which are isomorphic by
  \itemref{EI:locally-principal:sub}. Similarly,
  $\COHO{-1}(C) = \ker(L^\vee \to \Orb_X)$ and
  $\COHO{-1}(\LDERF c^*C) = \ker(L^\vee \tensor_{\Orb_X}
  \Orb_{\mathcal{X}} \to \Orb_{\mathcal{X}})$ are isomorphic by
  \itemref{E:locally-principal:tor}. In particular, both
  $\COHO{0}(\LDERF c^*C)$ and $\COHO{-1}(\LDERF c^*C)$ are
  quasi-coherent $\Orb_X$-modules. It follows immediately that
  $\LDERF c^*C$ is a quasi-coherent $\Orb_X$-module and so
  $\RDERF c_{\qcsubscript,*}\LDERF c^*C \simeq \RDERF c_*\LDERF c^*C
  \simeq C \tensor^{\LDERF}_{\Orb_X} \Orb_{\mathcal{X}}$ as
  $\Orb_X$-modules. The claim follows.

  We next prove \itemref{EI:locally-principle:faith}. Now condition
  \itemref{EI:locally-principal:sub} implies that
  $\MOD(Z) \simeq \MOD(\mathcal{Z})$, so
  $\DPCOH{-}(Z) \simeq \DPCOH{-}(\mathcal{Z})$. Since $C$ is perfect,
  Lemma \ref{L:projection_formula} and Proposition
  \ref{P:Zequi-torind}\itemref{PI:Zequi-torind:eta} together with
  \itemref{EI:locally-principle:eta} imply $c$ is faithful along
  $\LDERF c^*C$. Claim \itemref{EI:locally-principle:equi} is similar, so
  its proof is omitted.
\end{example}
The previous example can be generalized to the vanishing locus of a section of vector bundle.
\begin{example}\label{E:section-vb}
  Let $X$ be a quasi-compact and quasi-separated algebraic space. Let
  $F$ be a vector bundle on $X$ and let $s\in \Gamma(X,F)$. Let
  $i\colon Z \hookrightarrow X$ be the vanishing locus of $s$. Let
  $\Orb_{\mathcal{X}}$ be a sheaf of $\Orb_X$-algebras (not
  necessarily quasi-coherent) such that
  $K(s^\vee) \to K(s^\vee) \tensor_{\Orb_X}^{\LDERF}
  \Orb_{\mathcal{X}}$ is a quasi-isomorphism of $\Orb_X$-modules,
  where $K(s^\vee)$ is the Koszul complex associated to
  $s^\vee \colon F^\vee \to \Orb_X$ \cite[\S IV.2]{MR801033}. If $s$
  is a regular section of $F$, then this condition is equivalent to $s$
  remaining a regular section of
  $F\tensor_{\Orb_X} \Orb_{\mathcal{X}}$. As before, this condition is
  satisfied when $s$ is a regular section and $\Orb_{\mathcal{X}}$ is
  the formal completion of $\Orb_X$ along $I=\im(s^\vee)$. If $F$ is a
  line bundle, then
  $K(s^\vee) = [F^\vee \xrightarrow{s^\vee} \Orb_X]$. Hence, we see
  that this condition is equivalent to those in Example
  \ref{E:locally-principal}. Arguing as in Example
  \ref{E:locally-principal}, one can establish the following:
  \begin{enumerate}
  \item $c$ is faithful along $\LDERF c^*K(s^\vee)$; and
  \item if $\mathcal{N} \in \DPCOH{-}(\mathcal{X})$ is such that
    $\mathcal{N} \tensor^{\LDERF}_{\Orb_X} K(s^\vee) \in \DQCOH(X)$,
    then $c$ is an equivalence along $\LDERF c^*K(s^\vee)$ at
    $\mathcal{N}$.
  \end{enumerate}

\end{example}
\section{Lefschetz Theorems}\label{S:lefschetz}
To illustrate the strength of our reformulation, we can give a
brief proof of the following Lefschetz theorem.
\begin{theorem}\label{T:general-lefschetz}
  Let $X$ be a quasi-compact and quasi-separated algebraic space. Let
  $c\colon \mathcal{X} \to X_{\et}$ be a morphism of ringed topoi. Let
  $i\colon Z \to X$ be a closed immersion and let $r\geq 0$ be an integer. If
  \begin{enumerate}
  \item $U=X-Z$ is quasi-affine;
  \item $\RDERF \Gamma(X,\Orb_X) \to \RDERF \Gamma(\mathcal{X},\Orb_{\mathcal{X}})$ is $r$-connected; and
  \item there exists
    $\mathcal{M} \in \DCAT(\mathcal{X})$ such that
    \begin{enumerate}
    \item $\RDERF c_{\qcsubscript,*}\mathcal{M}$ is perfect with
      cohomological support $|Z|$, and
    \item  $c$ is faithful along
      $\mathcal{M}$ at $\Orb_{\mathcal{X}}$;
  \end{enumerate}
  \end{enumerate}
  then $\Orb_X \to \RDERF c_{\qcsubscript,*}\Orb_{\mathcal{X}}$ is
  $r$-connected. In particular, if $N\in \DQCOH(X)$ is perfect of tor-amplitude $\geq a$, 
  \[
    \RDERF \Gamma(X,N) \to \RDERF \Gamma(\mathcal{X},\LDERF c^*N)
  \]
  is $r+a$-connected.
\end{theorem}
\begin{proof}
  By the projection formula (Lemma \ref{L:projection_formula}),
  $\pi_{\Orb_X,\mathcal{M}}$ is an isomorphism. Thus, Proposition
  \ref{P:Zequi-faithful-use} implies that
  $\RDERF c_{\qcsubscript,*}\mathcal{M} \tensor \eta_{\Orb_X}$ is an
  isomorphism. Let $H$ be a cone for
  $\eta_{\Orb_X} \colon \Orb_X \to \RDERF
  c_{\qcsubscript,*}\Orb_{\mathcal{X}}$; then we have just proved that
  $H \tensor^{\LDERF}_{\Orb_X} \RDERF c_{\qcsubscript,*}\mathcal{M}
  \simeq 0$ (see \eqref{LEQ:key-projection-ff:2}). It remains to prove
  that $\trunc{\leq r}H \simeq 0$. Let $j \colon U \to X$ be the
  resulting open immersion. The theory of smashing Bousfield
  localizations implies immediately that
  $H \simeq \RDERF j_*\LDERF j^*H$ (e.g.,
  \cite[Ex.~1.4]{telescope-stacks}). 
  Now $\trunc{\leq -2}H \simeq 0$ and so
  \[
    \Gamma(U,\COHO{-1}(j^*H))\simeq \shfcoho^{-1}(U,j^*H) \simeq \shfcoho^{-1}(X,H) \simeq 0.
  \]
  Since $U$ is quasi-affine, $\COHO{-1}(j^*H) = 0$ and so
  $\COHO{-1}(H) \simeq j_*\COHO{-1}(j^*H) \simeq 0$. That is,
  $\trunc{\leq -1}H \simeq 0$. Repeating this argument, we obtain that
  $\trunc{\leq r}H \simeq 0$; that is,
  $\Orb_X \to \RDERF c_{\qcsubscript,*}\Orb_{\mathcal{X}}$ is
  $r$-connected. 
\end{proof}
In the tor-independent case, we have the following variant of Theorem \ref{T:general-lefschetz}.
\begin{corollary}\label{C:FF}
  Let $X$ be a quasi-compact and quasi-separated algebraic space. Let
  $c \colon \mathcal{X} \to X_{\et}$ be a morphism of ringed topoi. Let $i\colon Z \hookrightarrow X$ a closed immersion and let $r \geq 0$ be an integer. If
  \begin{enumerate}
  \item $U=X-Z$ is quasi-affine,
  \item   \label{MTI:dim_1FF:h0} the morphism
    $\RDERF \Gamma(X,\Orb_X) \to
    \RDERF \Gamma(\mathcal{X},\Orb_{\mathcal{X}})$ is $r$-connected,
    \item $c$ and $i$ are tor-independent,  
  \item $\Orb_X$ is coherent  or $i_*\Orb_Z$ is
    perfect, and
  \item $\DPCOH{-}(X_{\et},i_*\Orb_Z) \simeq \DPCOH{-}(\mathcal{X},c^*i_*\Orb_Z)$;
      \end{enumerate}
        then $\Orb_X \to \RDERF c_{\qcsubscript,*}\Orb_{\mathcal{X}}$ is
  $r$-connected. In particular, if $N \in \DQCOH(X)$ is perfect of tor-amplitude $\geq a$,
  \[
    \RDERF \Gamma(X,E) \to \RDERF \Gamma(\mathcal{X},c^*E)
  \]
  is $r+a$-connected.
\end{corollary}
\begin{proof}
  If $i_*\Orb_Z$ is perfect, then set
  $\mathcal{M} = \LDERF c^*i_*\Orb_Z$. If $\Orb_X$ is coherent, then
  set $\mathcal{M}=\LDERF c^*M$, where $M$ is as in Lemma
  \ref{C:Zequiv-examples-tor-ind-coherent}. Since $c$ and $i$ are
  tor-independent, Corollary
  \ref{C:Zequiv-examples-tor-ind}\itemref{LI:Zequiv-examples-tor-ind:perfect}
  (in the perfect case) and Corollary
  \ref{C:Zequiv-examples-tor-ind-coherent} (in the coherent case)
  imply that $c$ is faithful along $\mathcal{M}$. Moreover, Lemma
  \ref{L:Zequiv-locus} implies that
  $\RDERF c_{\qcsubscript,*}\mathcal{M} = \RDERF
  c_{\qcsubscript,*}\LDERF c^*M \simeq M$, which is perfect with
  cohomological support $|Z|$. The
  result now follows from Theorem \ref{T:general-lefschetz}.
\end{proof}
In the following theorem, we can optimize the above results
substantially in the case of a Cartier divisor, making them amenable
to an inductive process.
\begin{theorem}\label{T:lefschetz-cartier}
  Let $X$ be a quasi-compact and quasi-separated algebraic space. Let
  $c\colon \mathcal{X} \to X_{\et}$ be a morphism of ringed topoi. Let
  $L$ be a line bundle on $X$, $s\in \Gamma(X,L)$, and
  $i\colon Z \hookrightarrow X$ its vanishing locus. Let
  $r \geq 0$ be an integer. If
  \begin{enumerate}
  \item\label{TI:lefschetz-cartier:qaff} $U = X-Z$ is quasi-affine;
    \item\label{TI:lefschetz-cartier:glob}
    $\RDERF\Gamma(X,\Orb_X) \to \RDERF
    \Gamma(\mathcal{X},\Orb_{\mathcal{X}})$ is $r$-connected; and either 
  \item\label{TI:lefschetz-cartier:slice}
   \begin{enumerate} 
   \item $C \to \RDERF c_{\qcsubscript,*}\LDERF c^*C$ is $r$-connected,
    where $C=[L^\vee \xrightarrow{s^\vee} \Orb_X]$; or 
  \item \label{TI:lefschetz-cartier:ti} $c$ and $i$ are tor-independent and either
\begin{enumerate}
\item $\Orb_{Z} \to \RDERF c_{Z,\qcsubscript,*}\Orb_{\mathcal{Z}}$ is
    $r$-connected; or
    \item $Z$ is quasi-affine and $\RDERF \Gamma(Z,\Orb_Z) \to \RDERF \Gamma(\mathcal{Z},\Orb_{\mathcal{Z}})$ is $r$-connected;
   \end{enumerate}
    \end{enumerate}
  \end{enumerate}
  then $\Orb_X \to \RDERF c_{\qcsubscript,*}\Orb_{\mathcal{X}}$ is
  $r$-connected. In particular, if $N \in \DQCOH(X)$ is perfect of tor-amplitude $\geq a$,
  \[
    \RDERF \Gamma(X,E) \to \RDERF \Gamma(\mathcal{X},c^*E)
  \]
  is $r+a$-connected.
\end{theorem}
\begin{proof}
  Let $H_N$ be the cone of $\eta_{N}$ and set $H=H_{\Orb_X}$; it suffices to show that $\trunc{\leq r}H \simeq 0$. Since $C$ is perfect, $\pi_{C,\LDERF c^*\Orb_{{X}}}$ is an
    isomorphism (Lemma
  \ref{L:projection_formula}). In particular, 
    $H \tensor^{\LDERF}_{\Orb_X} C \simeq H_C$. By condition
    \itemref{TI:lefschetz-cartier:slice}, we conclude that
    $\trunc{\leq r}(H \tensor^{\LDERF}_{\Orb_X} C) \simeq
    0$. Now consider the distinguished triangle:
  \[
    \xymatrix{L^\vee \tensor^{\LDERF}_{\Orb_X} H \ar[r] & H \ar[r] & C
      \tensor^{\LDERF}_{\Orb_X} H \ar[r] & L^\vee \tensor^{\LDERF}_{\Orb_X}
      H[1].}
  \]
  But $L^\vee$ is a line bundle, so
  $\trunc{\leq r}(L^\vee \tensor^{\LDERF}_{\Orb_X} H) \simeq L^\vee
  \tensor^{\LDERF}_{\Orb_X} \trunc{\leq r}H$. It follows immediately
  from the distinguished triangle above that
  $C \tensor^{\LDERF}_{\Orb_X} \trunc{\leq r}H \simeq 0$. Let
  $j \colon U = X -Z \hookrightarrow X$ be the resulting open
  immersion, which is affine; then the theory of smashing Bousfield
  localizations implies immediately that
  $\trunc{\leq r}H \simeq \RDERF j_*\LDERF j^*\trunc{\leq r}H$ (e.g.,
  \cite[Ex.~1.4]{telescope-stacks}). Noting that
  $\trunc{\leq j}\RDERF\Gamma(X,H) \simeq \trunc{\leq j}\RDERF
  \Gamma(X,\trunc{\leq j}H)$ for all $j\in \Z$, now argue as in
  Theorem \ref{T:general-lefschetz}. 
\end{proof}
\section{Pseudo-conservation}
Let $\mathcal{X}$ be a ringed topos. We say that a collection
$S \subseteq \DCAT(\mathcal{X})$ is \emph{pseudo-conservative} if whenever
$\mathcal{M} \in \DPCOH{-}(\mathcal{X})$ satisfies
$\mathcal{M} \otimes^{\LDERF}_{\Orb_{\mathcal{X}}} \mathcal{Q} \simeq 0$ for all
$\mathcal{Q} \in S$, then $\mathcal{M} \simeq 0$.
\begin{example}\label{E:closed-points}
  Let $X$ be a locally ringed space or quasi-compact and quasi-separated algebraic space. Let
  $|X|_{\mathrm{cl}}$ be the set of closed points of $X$. The
  collection $\{ \kappa(x) \}_{x\in |X|_{\mathrm{cl}}}$ is
  pseudo-conservative. This is immediate from Nakayama's Lemma.
\end{example}
\begin{example}\label{E:closed-morphism}
  Let $A$ be a ring. Let $X \to \spec A$ be a quasi-compact and closed
  morphism of algebraic spaces. Let $I \subseteq A$ be an ideal
  contained in the Jacobson radical of $A$. Let
  $X_0 = X\times_{\spec A} \spec (A/I)$ and take $i \colon X_0 \to X$
  be the resulting closed immersion. Then $\{\Orb_{X_0}\}$ is
  pseudo-conservative. Indeed, if $M \in \DPCOH{-}(X)$ is non-zero,
  then its top cohomology group $\COHO{top}(M)$ is finitely
  generated. It follows that its support $W$ is a non-empty closed
  subset of $X$. Hence, the image of $W$ in $\spec A$ is closed and
  non-empty. Since $I$ is contained in the Jacobson radical of $A$,
  $W$ meets $\spec (A/I)$.
\end{example}
We have the following useful lemma. 
\begin{lemma}\label{L:top-cons}
  Let $\mathcal{X}$ be a ringed topos. Let $S \subseteq \DCAT^-(\mathcal{X})$ be a collection of
  objects. Consider 
  \[
    S' = \{ \COHO{t(\mathcal{Q})}(\mathcal{Q}) \suchthat \mathcal{Q} \in S\} \subseteq \MOD(\mathcal{X}),
  \]
  where $t(\mathcal{Q})$ denotes the top cohomological degree of
  $\mathcal{Q}$. If $\Orb_{\mathcal{X}}$ is coherent, then the following are equivalent:
  \begin{enumerate}
  \item $S$ is pseudo-conservative; 
  \item if $\mathcal{M} \in \COH(\mathcal{X})$ and
  $\mathcal{M} \tensor_{\Orb_X} \mathcal{Q} \cong 0$ for all
  $\mathcal{Q} \in S'$, then $\mathcal{M} \cong 0$. 
\end{enumerate}
\end{lemma}
\begin{proof}
  This is immediate from the following: if $\mathcal{M}$,
  $\mathcal{N} \in \DCAT^-(\mathcal{X})$, then
  \[
    \COHO{t(\mathcal{M})+t(\mathcal{Q})}(\mathcal{M}
    \tensor^{\LDERF}_{\Orb_{\mathcal{X}}} \mathcal{N}) \cong \COHO{t(\mathcal{M})}(\mathcal{M}) \tensor_{\Orb_{\mathcal{X}}} \COHO{t(\mathcal{N})}(\mathcal{N}).  \qedhere
  \]
\end{proof}

\section{GAGA}
In this section, we prove our general GAGA theorem. We will see in
\S\ref{S:applications} that this implies all
existing results in the literature for algebraic spaces. Given what we
have already established, its proof is straightforward.
\begin{theorem}\label{T:npGAGA}
  Let $X$ be a quasi-compact and quasi-separated algebraic space. Let
  $c\colon \mathcal{X} \to X_{\et}$ be a morphism of ringed topoi. Let
  $\Lambda$ index a family of $2$-commutative diagrams
  of ringed topoi:
  \[
    \xymatrix{\mathcal{Z}_\lambda \ar[r]^{c_\lambda}
      \ar[d]_{i_\lambda'} & Z_{\lambda,\et} \ar[d]^{i_\lambda} \\
      \mathcal{X} \ar[r]^c & X_{\et},}
  \]
  where $i_\lambda$ is quasi-affine for all $\lambda\in \Lambda$. For
  each $\lambda\in \Lambda$, let
  $\mathcal{M}_\lambda \in \langle \RDERF
  i'_{\lambda,*}\DPCOH{-}(\mathcal{Z}_\lambda)\rangle$.
  \begin{enumerate}[label=(\roman*), ref=\roman*]
  \item\label{TI:npGAGA:ff} Let $N\in \DPCOH{-}(X)$. Assume that
    $\RDERF c_{\qcsubscript,*}\LDERF c^*N \in \DPCOH{-}(X)$ and
    $\{\RDERF c_{\qcsubscript,*}\mathcal{M}_\lambda\}_{\lambda\in
      \Lambda}$ is pseudo-conservative. If $c$ is faithful along
    $\mathcal{M}_\lambda$ at $\LDERF c^*N$ for all
    $\lambda \in \Lambda$, then
    \[
      \eta_M \colon M \to \RDERF c_{\qcsubscript,*}\LDERF c^*M
    \]
    is an isomorphism. 
  \item\label{TI:npGAGA:ess} Let
    $\mathcal{N} \in \DPCOH{-}(\mathcal{X})$. Assume that
    $\RDERF c_{\qcsubscript,*}\mathcal{N} \in \DPCOH{-}(X)$ and
    $\{\mathcal{M}_{\lambda}\}_{\lambda\in \Lambda}$ is
    pseudo-conservative. If $c$ is an
    equivalence along $\mathcal{M}_\lambda$ at $\mathcal{N}$ for all $\lambda \in \Lambda$,
    then
    \[
      \epsilon_{\mathcal{M}} \colon \LDERF c^*\RDERF
      c_{\qcsubscript,*}\mathcal{M} \to \mathcal{M}
    \]
    is an isomorphism.
  \end{enumerate}
  In addition assume that $*\in \{b,-\}$ and
  \begin{enumerate}
  \item $X$ is proper and pseudo-coherent over an affine scheme
    $\spec A$;
  \item \label{TI:npGAGA:bounded} $\RDERF \Gamma(\mathcal{X},-)$ sends $\DPCOH{*}(\mathcal{X})$
    to $\DPCOH{*}(A)$; and
  \item if $*=b$, then $\LDERF c^*$ sends $\DPCOH{b}(X)$ to $\DPCOH{b}(\mathcal{X})$.
  \end{enumerate}
  If
  $\{\RDERF c_{\qcsubscript,*}\mathcal{M}_\lambda \}_{\lambda\in
    \Lambda}$ (resp.~$\{\mathcal{M}_\lambda\}_{\lambda\in \Lambda}$)
  is pseudo-conservative and $c$ is
  faithful (resp.~an equivalence) by
  $\mathcal{M}_\lambda$ for all $\lambda\in \Lambda$, then
  \[
    \LDERF c^* \colon \DPCOH{*}(X) \to \DPCOH{*}(\mathcal{X})
  \]
  is fully faithful (resp.~essentially surjective).
\end{theorem}
\begin{proof}
  For \itemref{TI:npGAGA:ff}, by Lemma \ref{L:Zequiv-locus} and Proposition
  \ref{P:Zequi-faithful-use}, we have that
  $ \RDERF c_{\qcsubscript,*} \mathcal{M}_\lambda \tensor \eta_N$ is an
  isomorphism for all $\lambda\in \Lambda$. But $N$ and
  $\RDERF c_{\qcsubscript,*}\LDERF c^*N \in \DPCOH{-}(X)$, so
  $\cone(\eta_N) \in \DPCOH{-}(X)$. Since
  $\{\RDERF c_{\qcsubscript,*}\mathcal{M}_\lambda\}_{\lambda\in
    \Lambda}$ is pseudo-conservative, the claim follows.

  For \itemref{TI:npGAGA:ess}, 
  $\mathcal{M}_\lambda \tensor \epsilon_{\mathcal{N}}$ is an
  isomorphism for all $\lambda\in \Lambda$. But $\mathcal{N}$, 
  $\LDERF c^*\RDERF c_{\qcsubscript,*}\mathcal{N} \in
  \DPCOH{-}(\mathcal{X})$, so
  $\cone(\epsilon_{\mathcal{N}}) \in \DPCOH{-}(\mathcal{X})$. Since
  $\{\mathcal{M}_\lambda\}_{\lambda\in \Lambda}$ is
  pseudo-conservative, the claim follows.

  The last claim is immediate from the above and Proposition
  \ref{P:adjoints}.
\end{proof}
\begin{remark}
  In Theorem \ref{T:npGAGA}\itemref{TI:npGAGA:ff}, if we do not assume
  that $\RDERF c_{\qcsubscript,*}\LDERF c^*N$ belongs to
  $\DPCOH{-}(X)$, but the
  $\{\RDERF c_{\qcsubscript,*}\mathcal{M}_\lambda\}_{\lambda\in
    \Lambda}$ are \emph{conservative} (i.e., if $Q \in \DQCOH(X)$ and
  $Q \tensor^{\LDERF}_{\Orb_X} \RDERF
  c_{\qcsubscript,*}\mathcal{M}_\lambda \simeq 0$ for all
  $\lambda\in \Lambda$ implies that $Q \simeq 0$), then we get the
  same conclusion.
\end{remark}
\section{Applications}\label{S:applications}
We begin with the following tor-independent refinement of Theorem \ref{T:npGAGA}. 
\begin{theorem}\label{T:ncGAGA}
  Let $A$ be a ring. Let $\pi \colon X \to \spec A$ be a proper and
  pseudo-coherent morphism of algebraic spaces. Let
  $c\colon \mathcal{X} \to X_{\et}$ be a morphism of ringed
  topoi.  Let
  $\Lambda$  index a family of $2$-commutative diagrams
  of ringed topoi:
  \[
    \xymatrix{\mathcal{Z}_\lambda \ar[r]^{c_\lambda}
      \ar[d]_{i_\lambda'} & Z_{\lambda,\et} \ar[d]^{i_\lambda} \\
      \mathcal{X} \ar[r]^c & X_{\et},}
  \]
  where $i_\lambda$ is affine for all $\lambda\in \Lambda$. Let $*\in \{b,-\}$. Assume that
  \begin{enumerate}[label=(\alph*), ref=\alph*]
  \item\label{TI:ncGAGA:bounded}
    $\RDERF \Gamma(\mathcal{X},-)$ sends
    $\DPCOH{*}(\mathcal{X})$ to $\DPCOH{*}(A)$;
  \item\label{TI:ncGAGA:flat} for all $\lambda\in \Lambda$, $c$ and
    $i_\lambda$ are tor-independent;
  \item for all $\lambda\in \Lambda$, $\DPCOH{-}(\mathcal{Z}_\lambda) \simeq \DPCOH{-}(\mathcal{X},i'_{\lambda,*}\Orb_{\mathcal{Z}_\lambda})$; 
  \item\label{TI:ncGAGA:Zequiv} for all $\lambda\in \Lambda$,
    $\LDERF c_\lambda^* \colon\DPCOH{-}(Z_\lambda) \to
    \DPCOH{-}(\mathcal{Z}_\lambda)$ is an equivalence;
  \item\label{TI:ncGAGA:btor} if $*=b$, then $\LDERF c^*$ sends
    $\DPCOH{b}(X)$ to $\DPCOH{b}(\mathcal{X})$;
  \item\label{TI:ncGAGA:projection} $\Orb_X$ is coherent; or
    $\Orb_{\mathcal{X}}$ is a compact object of $\DCAT(\mathcal{X})$;
    or for all $\lambda\in \Lambda$, $i_{\lambda,*}\Orb_{Z_\lambda}$
    is perfect.
  \end{enumerate}
  If $\{\Orb_{Z_{\lambda}}\}_{\lambda\in \Lambda}$
  (resp.~$\{\Orb_{\mathcal{Z}_\lambda}\}_{\lambda\in \Lambda}$) is
  pseudo-conservative, then
  \[
    \LDERF c^* \colon \DPCOH{*}(X) \to \DPCOH{*}(\mathcal{X})
  \]
  is fully faithful (resp.~essentially surjective).
\end{theorem}
\begin{proof}
  In the case where $\Orb_{\mathcal{X}}$ is a compact object of
  $\DCAT(\mathcal{X})$ or $i_{\lambda,*}\Orb_{Z_\lambda}$ is perfect
  for all $\lambda\in \Lambda$, the result is immediate from
  Theorem \ref{T:npGAGA} and Corollary
  \ref{C:Zequiv-examples-tor-ind}. In the case where $\Orb_X$ is
  coherent, we apply Corollary
  \ref{C:Zequiv-examples-tor-ind-coherent} to produce a perfect
  complex
  $M_\lambda \in \langle \RDERF i_*\DPCOH{-}(Z_\lambda)\rangle$ with
  $c$ an equivalence along $\LDERF c^*M_\lambda$ for all
  $\lambda\in \Lambda$. But if
  $\{\Orb_{Z_\lambda}\}_{\lambda\in \Lambda}$
  (resp.~$\{\Orb_{\mathcal{Z}_\lambda}\}_{\lambda\in \Lambda}$) is
  pseudo-conservative, then $\{M_\lambda\}_{\lambda\in \Lambda}$
  (resp.~$\{\LDERF c^*M_\lambda\}_{\lambda\in \Lambda}$) is
  pseudo-conservative (Lemma \ref{L:top-cons}). Now apply Theorem
  \ref{T:npGAGA}.
\end{proof}
It is easy to use Theorems \ref{MT:A}, \ref{T:npGAGA}, and
\ref{T:ncGAGA} to prove existing GAGA results.
\begin{example}[Analytic spaces]\label{E:anGAGA}
  Let $X \to \spec \C$ be a proper scheme. Let
  $c\colon X_{\mathrm{an}} \to X$ be its complex analytification. Now
  $X_{\mathrm{an}}$ is a Hausdorff topological space and $c$ is
  bijective on closed points; indeed $|X_{\mathrm{an}}|=X(\spec
  \C)$. Also, the local rings of $\Orb_{X_{\mathrm{an}}}$ are
  noetherian and the induced morphism
  $\Orb_{X,c(x)} \to \Orb_{X_{\mathrm{an}},x}$ is an isomorphism on
  maximal-adic completions \cite[XII.1.1]{SGA1}. By Remark \ref{R:MT:A},
  we see that conditions \itemref{GTI:A:bijective},
  \itemref{GTI:A:support}, and \itemref{GTI:A:flatunr_along_S} are satisfied. Cartan--Serre \cite{cartan-serre_finiteness} (also see 
  \cite[10.5.6]{MR755331}) gives condition
  \itemref{GTI:A:bd}. Condition \itemref{GTI:A:coh} is 
  Oka's Coherence Theorem (see \cite{oka_coherence} and \cite[2.5.3]{MR755331}). By Theorem \ref{MT:A}
  we may conclude that if $F\in \COH(X)$, then
  \[
    \shfcoho^{i}(X,F) \simeq \shfcoho^i(X_{\mathrm{an}},F_{\mathrm{an}})
  \]
  and $c^* \colon \COH(X) \to \COH(X_{\mathrm{an}})$ is an
  equivalence.
\end{example}
\begin{example}[Formal GAGA]
  Let $X \to \spec R$ be a proper morphism of schemes. Assume that $R$
  is noetherian. Let $I \subseteq R$ be an ideal and assume that $R$
  is complete with respect to the $I$-adic topology. Let
  $c\colon \hat{X} \to X$ be the formal completion of $X$ along the
  closed subscheme $X_0=X\otimes_R (R/I)$. It is easily verified using
  the results of \cite[III${}_1$]{EGA} that $c$ satisfies the
  hypotheses of Theorem \ref{MT:A}. Hence, we have the cohomological
  comparison result and the equivalence on categories of sheaves. It
  is also easy to use these arguments and Theorem
  \ref{T:ncGAGA} to prove formal GAGA for proper
  algebraic spaces. We again leave this as an exercise to the
  reader. One can also use these arguments to prove the formal
  GAGA statements of \cite{fujiwara-kato-I}, which hold for
  certain non-noetherian base rings $A$ (e.g., $A$ is the $a$-adic
  completion of a finitely presented $V$-algebra, where $V$ is an
  $a$-adically complete valuation ring.).
\end{example}
\begin{example}[Rigid GAGA]
  Let $X \to \spec R$ be a proper morphism of schemes. Let $k$ be a
  complete nonarchimean field. Assume that $R$ is an affinoid
  $k$-algebra; that is, it is a Banach $k$-algebra that is a quotient
  of some Tate algebra
  $T_n=k\langle\!\langle Y_1,\dots, Y_n \rangle\!\rangle$, where $Y_n$
  is the subalgebra of $k[[Y_1,\dots,Y_n]]$ consisting of power series
  that are convergent with respect to the Gauss norm (i.e., suprema of
  coefficients). Associated to $X$ is a natural morphism of locally
  ringed $G$-spaces $c \colon X_{\mathrm{rig}} \to X$, where
  $X_{\mathrm{rig}}$ is a rigid analytic space. The underlying
  topological space of $X_{\mathrm{rig}}$ is Hausdorff and its points
  correspond to closed points of $X$. Moreover,
  $\Orb_{X_{\mathrm{rig}}}$ is a coherent sheaf with noetherian local
  rings. Also, $R$ is noetherian. Kiehl's Finiteness Theorem
  \cite{MR0210948} implies that the cohomology of coherent sheaves on
  $\Orb_{X_{\mathrm{rig}}}$ satisfies the condition \itemref{GTI:A:bd}
  of Theorem \ref{MT:A}. Again, we get the cohomological comparison
  result and equivalence on categories of coherent sheaves. Using
  \cite{MR2524597}, one can make sense of rigid analytifications of
  separated algebraic spaces. This allows one to prove rigid GAGA in
  this context too. One can also prove adic and Berkovich GAGA
  statements using this method.
\end{example}
\begin{example}[Non-noetherian formal GAGA]\label{E:non-noetherian-gaga}
  Here we will use Theorem \ref{T:npGAGA} to prove the
  GAGA result in the Stacks Project \cite[Tag
  \spref{0DIA}]{stacks-project}. The situation is as in Example \ref{E:Zequiv-non-noetherian}, and it is immediate from \ref{T:npGAGA} that we obtain an equivalence
  \[
    \LDERF c^* \colon \DPCOH{-}(X) \to \DPCOH{-}(\mathcal{X}).
  \]
\end{example}
The following is a variant of the results established in
\cite[\S1]{MR3250065}. It is a simple consequence of Theorem
\ref{T:lefschetz-cartier}.
\begin{theorem}\label{T:lefschetz-dadic}
  Let $X$ be a quasi-compact and quasi-separated algebraic
  space. Let $i \colon D \subseteq X$ be a Cartier divisor. Let
  $c\colon \hat{X} \to X$ be the $D$-adic completion. If
  \begin{enumerate}
  \item \label{TI:lefschetz-dadic:qaff} $X-D$ is quasi-affine;
  \item \label{TI:lefschetz-dadic:h0}
    $\shfcoho^0(X,\Orb_X) \simeq \varprojlim_r \shfcoho^0(X,\Orb_{D_r})$; and
  \item \label{TI:lefschetz-dadic:h1}
    $\mathrm{H}^1(X,\Orb_X(-rD)) = 0$ for all $r \gg 0$;
  \end{enumerate}
  then the comparison morphism:
  \[
    \shfcoho^0(X,E) \to \shfcoho^0(\hat{X},\hat{E})
  \]
  is an isomorphism for all vector bundles $E$ on $X$.
\end{theorem}
\begin{proof}
  We simplify verify that the hypotheses of Theorem
  \ref{T:lefschetz-cartier} are satisfied. General properties of
  completion say that $c$ and $i$ are tor-independent \cite[Tag
  \spref{0BNG}]{stacks-project}. By Example \ref{E:Zequiv}, it remains
  to prove that
  $\shfcoho^i(X,\Orb_X) \to \shfcoho^i(\hat{X},\Orb_{\hat{X}})$ is an
  isomorphism when $i=0$ and is injective when $i=1$. The $i=0$ condition
  is immediate from \itemref{TI:lefschetz-dadic:h0}. For the $i=1$
  case, we have the Milnor exact sequence:
  \[
    \xymatrix{0 \ar[r] & \varprojlim_n^1 \shfcoho^0(X,\Orb_{D_n})
      \ar[r] & \shfcoho^1(\hat{X},\Orb_{\hat{X}}) \ar[r] &
      \varprojlim_n \shfcoho^1(X,\Orb_{D_n}) \ar[r] & 0.}
  \]
  It follows from \itemref{TI:lefschetz-dadic:h1} that
  $\{ \shfcoho^0(X,\Orb_{D_n})\}_{n\geq 0}$ is eventually a surjective
  system, so is Mittag--Leffler. In particular, $\varprojlim^1_n$
  vanishes. Hence,
  $\shfcoho^1(\hat{X},\Orb_{\hat{X}}) \simeq \varprojlim_n
  \shfcoho^1(X,\Orb_{D_n})$. But \itemref{TI:lefschetz-dadic:h1} also
  gives
  $\shfcoho^1(X,\Orb_X) \hookrightarrow \shfcoho^1(X,\Orb_{D_n})$ for
  $n\gg 0$.  Since inverse limits are left exact, this gives
  $\shfcoho^1(X,\Orb_X) \hookrightarrow \varprojlim_n
  \shfcoho^1(X,\Orb_{D_n})$; the result follows.
\end{proof}
\begin{remark}\label{R:lefschetz-dadic}
  Theorem
  \ref{T:lefschetz-dadic}\itemref{TI:lefschetz-dadic:h0} is implied
  by $\shfcoho^0(X,\Orb_X(-rD)) = 0$ for all $r \gg 0$.
\end{remark}
\begin{example}\label{E:leschetz-qaff}
  Let $X$ be a quasi-affine scheme. Let $A=\Gamma(X,\Orb_X)$ and
  $a\in A$. Let $D=\spec(A/a) \cap X \subseteq X$ and let $c \colon \hat{X} \to X$
  be the formal completion of $X$ along $D$. Assume
  \begin{enumerate}
  \item \label{CI:leschetz-qaff:complete} $A$ is $a$-adically complete;
  \item $a$ is not a zero divisor of $A$; and
  \item \label{CI:leschetz-qaff:sepH1}
    $a^N \shfcoho^1(X,\Orb_{X}) = 0$ for some $N\geq 1$.
  \end{enumerate}
  Then $\shfcoho^0(X,E) \to \shfcoho^0(\hat{X},\hat{E})$ is an
  isomorphism for all vector bundles $E$. This follows immediately
  from Theorem \ref{T:lefschetz-dadic}. For example, let
  $(A,\mathfrak{m})$ be an $\mathfrak{m}$-adically complete noetherian
  local ring and $X=\spec A - \{\mathfrak{m}\}$. Let
  $a\in \mathfrak{m}$ be a non-zero-divisor and let
  $c \colon \hat{X} \to X$ be the $a$-adic completion of $X$ (as a
  formal scheme). If $\mathrm{depth}_{\mathfrak{m}}(A/a) \geq 2$
  (e.g., $A/a$ is $S_2$ or normal or Cohen--Macaulay), then the
  conditions are satisfied. 
\end{example}
\begin{example}
  Let $X$ be an $S_2$-variety over a field $k$ (e.g. normal or
  smooth). Assume $X$ is projective and $D \subseteq X$ is an ample
  divisor. Let $c \colon \hat{X} \to X$ be the $D$-adic
  completion. Then $c^* \colon \Vect(X) \to \Vect(\hat{X})$ is fully
  faithful. Since $D$ is ample and $X$ is projective, $X-D$ is
  affine. Theorem \ref{T:lefschetz-dadic} and Remark
  \ref{R:lefschetz-dadic} now show it is sufficient to prove that
  $\shfcoho^i(X,\Orb_X(-rD)) = 0$ for $r\gg 0$ and $i=0$, $1$. The
  $i=0$ case is \cite[Tag \spref{0FD7}]{stacks-project} and the $i=1$
  case is Enriques--Severi--Zariski vanishing \cite[Tag
  \spref{0FD8}]{stacks-project}. This is closely related to the
  Lefschetz Hyperplane Theorem.
\end{example}
\appendix
 \section{The projection formula}\label{A:projection-formula}
We recall some results on the projection formula. An excellent source is \cite{MR1988072}.

Let $(\mathcal{C},\tensor, \alpha,\sigma,1,\lambda,\rho)$ be a symmetric monoidal category. That is,
\begin{itemize}
\item $\mathcal{C}$ is a category;
\item $-\tensor - \colon \mathcal{C} \times \mathcal{C}
  \to \mathcal{C}$ is a functor;
\item for each triple $x$, $y$, $z\in \mathcal{C}$ there is a functorial isomorphism
  \[
    \alpha_{x,y,z} \colon (x\tensor y)\tensor z \simeq x \tensor(y \tensor z),
  \]
  which satisfies the pentagram law;
\item for each pair $x$, $y\in \mathcal{C}$ there is a functorial isomorphism
  \[
    \sigma_{x,y} \colon x\tensor y \simeq y\tensor x 
  \]
  such that $\sigma_{y,x}\circ \sigma_{x,y} =
  \ID{}$, and is compatible with $\alpha$ in the obvious sense;
\item $1\in \mathcal{C}$; 
\item for each $x\in \mathcal{C}$ there are functorial isomorphisms
  \[
    \lambda_x \colon 1 \tensor x \simeq x \quad \mbox{and} \quad \rho_x \colon x \tensor 1 \simeq x,
  \]
  such that $\lambda_1=\rho_1$ and are compatible with $\alpha$ and
  $\sigma$ in the obvious sense.
\end{itemize}
Typically, we will just denote this data by $\mathcal{C}$. For background material on symmetric monoidal categories, we refer the interested reader to \cite{MR0338002,MR1712872,MR0170925,MR213412}.

Let $L \colon \mathcal{C} \to \mathcal{D}$ be a functor, where $\mathcal{C}$ and $\mathcal{D}$ are symmetric monoidal categories. We say $L$ is \emph{lax monoidal} if for each $c_1$, $c_2\in \mathcal{C}$ there is a natural morphism:
\[
  \mu_{c_1,c_2} \colon L(c_1) \tensor_{\mathcal{D}} L(c_2) \to L(c_1 \tensor_{\mathcal{C}} c_2)
\]
and a morphism
\[
  \iota \colon 1_{\mathcal{D}} \to L(1_{\mathcal{C}})
\]
that is all compatible with the monoidal structures in the obvious way. We denote this package of data by $(L,\mu,\iota)$. If $\mu_{c_1,c_2}$ and $\iota$ are always isomorphisms, then we say that $L$ is \emph{strong monoidal}. We have the following trivial lemma.
\begin{lemma}\label{L:composition-monoidal}
  Consider a sequence of lax monoidal functors between symmetric monoidal categories:
  \[
    \mathcal{C}_1 \xrightarrow{(L_1,\mu^{L_1},\iota^{L_1})} \mathcal{C}_2
    \xrightarrow{(L_2,\mu^{L_2},\iota^{L_2})} \cdots
    \xrightarrow{(L_n,\mu^{L_n},\iota^{L_n})} \mathcal{C}_{n+1}.
  \]
  \begin{enumerate}
  \item Then
    $L_n\cdots L_1 \colon \mathcal{C}_1 \to \mathcal{C}_{n+1}$ is lax
    monoidal via the inductively defined:
    \begin{align*}
      \mu^{L_n\cdots L_1}_{c_1,c_2} &= L_n(\mu^{L_{n-1}\cdots L_{1}}_{L_{n-1}\cdots L_1(c_1),L_{n-1}\cdots L_1(c_2)}) \\ 
      \iota^{L_n\cdots L_1} &= L_n(\iota^{L_{n-1}\cdots L_{1}}) 
    \end{align*}
    If the $L_i$ are all strong monoidal, then so too is the
    composition $L_n\cdots L_1$.
  \item For $j \geq i$ consider another sequence of lax monoidal functors:
    \[
      \mathcal{C}_i \xrightarrow{(L_i',\mu^{L_i'},\iota^{L_i'})} \mathcal{C}_{i+1}
    \xrightarrow{(L_{i+1}',\mu^{L_{i+1}'},\iota^{L_{i+1}'})} \cdots
    \xrightarrow{(L_{j}',\mu^{L_{j}'},\iota^{L_{j}'})} \mathcal{C}_{j+1}
    \]
    together with a natural transformation
    \[
      \kappa \colon (L_j\cdots L_i,\mu^{L_j\cdots L_i},\iota^{L_j
        \cdots L_{i}}) \Rightarrow (L_j' \cdots L_i',\mu^{L_j'\cdots
        L_i'} ,\iota^{L_j'\cdots L_i'}).
    \]
    Then there is a natural extension of $\kappa$ to a natural transformation:
    \begin{align*}
      (L_n &\cdots L_1, \mu^{L_n\cdots
      L_1},\iota^{L_n\cdots L_1}) \\
      &\Rightarrow (L_n \cdots L_{j+1}
      L_j'\cdots L_i'L_{i-1} \cdots L_1, \mu^{L_n \cdots L_{j+1}
      L_j'\cdots L_i'L_{i-1} \cdots L_1}, \iota^{L_n \cdots L_{j+1}
      L_j'\cdots L_i'L_{i-1} \cdots L_1}).
    \end{align*}    
  \end{enumerate}
\end{lemma}
Now assume that $L \colon \mathcal{C} \to \mathcal{D}$ is strong monoidal and consider a right adjoint
\[
  R \colon \mathcal{D} \to
  \mathcal{C}.
\]
If $c \in \mathcal{C}$ and $d \in \mathcal{D}$, then we have the resulting unit/counit morphisms
\[
  \eta_c \colon c \to RL(c) \quad
  \mbox{and} \quad \epsilon_{d} \colon LR(d) \to d.
\]
If $d_1$, $d_2 \in \mathcal{D}$, then there is a natural conjugate of $\mu_{c_1,c_2}$,
\[
  \nu_{d_1,d_2} \colon R(d_1)\tensor_{\mathcal{C}} R(d_2) \to R(d_1\tensor_{\mathcal{D}} d_2).
\]
It is obtained as the adjoint to the composition:
\[
  L(R(d_1) \tensor_{\mathcal{C}} R(d_2)) \xrightarrow{L(\mu_{R(d_1),R(d_2)}^{-1})} LR(d_1) \tensor_{\mathcal{D}} LR(d_2) \xrightarrow{\epsilon_{d_1} \tensor \epsilon_{d_2}} d_1 \tensor_{\mathcal{D}} d_2.
\]
Similarly, there is a conjugate to $\iota$:
\[
  \jmath \colon 1_{\mathcal{C}} \to R(1_{\mathcal{D}}).
\]
It is obtained as the adjoint to $\iota^{-1} \colon L(1_{\mathcal{C}}) \to 1_{\mathcal{D}}$. It is easily verified that $(R,\nu,\jmath)$ is lax monoidal. 

Now if $c\in \mathcal{C}$ and $d\in \mathcal{D}$, then there is a natural \emph{projection morphism}
\[
  \pi_{c,d} \colon c \tensor_{\mathcal{C}} R(d)  \to R(L(c)\tensor_{\mathcal{D}} d).
\]
Indeed, it is given as the composition:
\begin{equation}
 c \tensor_{\mathcal{C}}   R(d)  \xrightarrow{\eta_c \tensor \ID{R(d)}} RL(c) \tensor_{\mathcal{C}} R(d) \xrightarrow{\nu_{L(c),d}} R(L(c) \tensor_{\mathcal{D}} d). \label{eq:projection}
\end{equation}
\begin{remark}\label{R:projection-equivalence}
Note that if $L$ is an equivalence, then $\pi_{c,d}$ is an isomorphism.
\end{remark}

There is another way to produce a projection morphism
\[
  \tilde{\pi}_{c,d} \colon c \tensor_{\mathcal{C}} R(d)  \to R(L(c)\tensor_{\mathcal{D}} d).
\]
It can be given as the adjoint to the composition:
\[
  L(c \tensor_{\mathcal{C}} R(d)) \xrightarrow{\mu_{c,R(d)}^{-1}} L(c)
  \tensor_{\mathcal{D}} LR(d) \xrightarrow{\ID{L(c)} \tensor \epsilon_d}
  L(c) \tensor_{\mathcal{D}} d.
\]
We wish to point out that $\nu$ (and so consequently $\pi$) depend on the choice of the right adjoint $R$. Occasionally, it will be useful to observe this, and we do so by using a suitable superscript (e.g., $\nu^{L,R}$). 
\begin{lemma}\label{L:projections-equal}
  $\pi_{c,d} = \tilde{\pi}_{c,d}$. 
\end{lemma}
\begin{proof}
  The adjoint to $\pi_{c,d}$ factors as: 
  \[
    \small{\xymatrix@C+0.9pc{L(c\tensor_{\mathcal{C}} R(d)) \ar[r]^-{L(\eta_c \tensor \ID{})} & L(RL(c) \tensor_{\mathcal{C}} R(d)) \ar[r]^-{L(\mu_{RL(c),R(d)}^{-1})} & LRL(c) \tensor_{\mathcal{D}} LR(d) \ar[r]^-{\epsilon_{L(c)} \tensor \epsilon_d} & L(c) \tensor_{\mathcal{D}} d.}}
  \]
  The following square also commutes, by naturality:
  \[
    \xymatrix@C+1.5pc{L(c \tensor_{\mathcal{C}} R(d)) \ar[d]_{\mu_{c,R(d)}^{-1}} \ar[r]^-{L(\eta_c \tensor \ID{})} & L(RL(c) \tensor_{\mathcal{C}} R(d)) \ar[d]^{L(\mu_{RL(c),R(d)}^{-1})}\\
      L(c) \tensor_{\mathcal{D}} LR(d) \ar[r]^{L(\eta_c) \tensor
        \ID{}}& LRL(c) \tensor_{\mathcal{D}} LR(d).}
  \]
  Hence, the adjoint to $\pi_{c,d}$ factors as:
  \[
   \xymatrix@C+0.6pc{L(c \tensor_{\mathcal{C}} R(d)) \ar[r]^-{\mu_{c,R(d)}^{-1}} & L(c) \tensor_{\mathcal{D}} LR(d) \ar[r]^-{L(\eta_c) \tensor \ID{}} & LRL(c) \tensor_{\mathcal{D}}LR(d) \ar[r]^-{\epsilon_{L(c)} \tensor \epsilon_d} & L(c) \tensor_{\mathcal{D}} d.}
 \]
 By the unit/conuit equations for adjunction, the composition of the last two morphisms results in $\ID{} \tensor \epsilon_d$. 
The result now follows. 
\end{proof}
\begin{remark}\label{R:projection-natural-iso}
  If
  $\kappa \colon (L, \mu,\iota) \Rightarrow (L',\mu',
  \iota')$ is a natural transformation of strong monoidal functors,
  then there is a canonically induced natural transformation
  $\kappa^\vee \colon R' \Rightarrow R$ between chosen right
  adjoints. If $c\in \mathcal{C}$ and $d\in \mathcal{D}$, then it is
  easily verified from Lemma \ref{L:projections-equal} that the
  following diagram commutes:
  \[
    \xymatrix@C+2pc{c \tensor_{\mathcal{C}} R'(d) \ar[r]^{\pi^{L',R'}_{c,d}} \ar[d]_{\ID{} \tensor \kappa^\vee} & R'(L'(c) \tensor_{\mathcal{D}} d) \ar[dr]^{\kappa^\vee} &\\
      c \tensor_{\mathcal{C}} R(d) \ar[r]_{\pi^{L,R}_{c,d}} & R(L(c) \tensor_{\mathcal{D}} d) \ar[r]_{R(\kappa)} & R(L'(c) \tensor_{\mathcal{D}} d).}
  \]
\end{remark}
In the following lemma we record some useful commutative diagrams.
\begin{lemma}\label{L:key-projection-ff}
  Let $c$, $x\in \mathcal{C}$ and $d\in
  \mathcal{D}$. The following diagrams commute.
    \begin{align}
      &\xymatrix@C+1.5pc{L(c \tensor_{\mathcal{C}} R(d))
      \ar[d]_{L(\pi_{c,d})} & \ar[l]_{\mu_{c,R(d)}}  L(c)
                                                    \tensor_{\mathcal{D}} LR(d) \ar[d]^{\ID{}\tensor \epsilon_d} \\
      LR(L(c) \tensor d) \ar[r]^{\epsilon_{L(c) \tensor d}} & L(c)
                                                              \tensor_{\mathcal{D}} d.}\label{LEQ:key-projection-ff:ess}                                                      \\
            &\xymatrix{c \tensor_{\mathcal{C}} x \ar[r]^{\eta_{c \tensor x}}
        \ar[d]_{c \tensor \eta_x} & RL(c \tensor_{\mathcal{C}} x)
         \\ c \tensor_{\mathcal{C}} RL(x)
        \ar[r]_-{\pi_{c,L(x)}} & R(L(c) \tensor_{\mathcal{D}} L(x)) \ar[u]_-{R(\mu_{c,x})}. }  \label{LEQ:key-projection-ff:2}
    \end{align}
\end{lemma}
\begin{proof}
  The diagram \eqref{LEQ:key-projection-ff:ess} is just a restatement
  of Lemma \ref{L:projections-equal}. For the commutativity of 
  \eqref{LEQ:key-projection-ff:2}, we observe that the adjoint of the composition going right
  and then down is simply $\mu_{c,x}^{-1}$. Working the other way, we see that the adjoint map is the composition:
  \[
 \xymatrix@C+0.8pc{ L(c \tensor_{\mathcal{C}} x) \ar[r]^-{L(\ID{}\tensor \eta_x)} & L(c \tensor_{\mathcal{C}} RL(x)) \ar[r]^-{\mu_{c,RL(x)}^{-1}} &  L(c) \tensor_{\mathcal{D}} LRL(x) \ar[r]^-{\ID{} \tensor \epsilon_{L(x)}} & L(c) \tensor_{\mathcal{D}} L(x).}
  \]
  By functoriality and naturality, the following diagram commutes:
  \[
    \xymatrix@C+1pc{L(c\tensor_{\mathcal{C}} x) \ar[r]^{\mu_{c,x}^{-1}} \ar[d]_{L(\ID{}\tensor \eta_x)} & L(c) \tensor_{\mathcal{D}} L(x) \ar[d]^{\ID{} \tensor L\eta_{x}}\\ L(c \tensor_{\mathcal{C}} RL(x)) \ar[r]^{\mu_{c,RL(x)}^{-1}} & L(c) \tensor_{\mathcal{D}} LRL(x).}
  \]
  It follows that the map we are interested in is actually the composition:
  \[
    \xymatrix@C+0.8pc{L(c \tensor_{\mathcal{C}} x) \ar[r]^-{\mu_{c,x}^{-1}} & L(c)
      \tensor_{\mathcal{D}} L(x) \ar[r]^-{\ID{}\tensor L\eta_x} & L(c)
      \tensor_{\mathcal{D}} LRL(x) \ar[r]^-{\ID{}\tensor
        \epsilon_{L(x)}} & L(c) \tensor_{\mathcal{D}} L(x). }
  \]
  By the unit/counit equations for the adjunction, the final two
  morphisms compose to give the identity. The result follows. 
\end{proof}
The following two lemmas establish the functoriality properties of the projection morphism.
\begin{lemma}\label{L:projection-composition}
  Consider strong monoidal functors:
  \[
    \mathcal{C} \xrightarrow{(L,\mu^L,\iota^L)} \mathcal{D}
    \xrightarrow{(S,\mu^S,\iota^S)} \mathcal{D}'.
  \]
  Assume that $L$ and $S$ admit right adjoints $R$ and $T$, respectively. 
  \begin{enumerate}
    \item \label{LI:projection-composition:conj} The conjugate
      (via $RT$) to $\mu^{SL}_{c_1,c_2}$ is the composition:
    \[
    \nu^{SL,RT}_{d_1',d_2'} \colon  RT(d_1') \tensor_{\mathcal{C}} RT(d_2') \xrightarrow{\nu_{T(d_1'),T(d_2')}^{L,R}} R(T(d_1') \tensor_{\mathcal{D}} T(d_2')) \xrightarrow{R(\nu_{d_1',d_2'}^{S,T})} RT(d_1' \tensor_{\mathcal{D}'} d_2').
    \]
  \item \label{LI:projection-composition:proj} If $c\in \mathcal{C}$ and $d'\in \mathcal{D}'$, then the
    following diagram commutes:
    \[
      \xymatrix@C+3pc{c \tensor_{\mathcal{C}} RT(d')
        \ar[r]^{\pi^{L,R}_{c,T(d')}} \ar[dr]_{\pi^{SL,RT}_{c,d'}} & R(L(c)
        \tensor_{\mathcal{D}} T(d')) \ar[d]^{R(\pi^{S,T}_{Lc,d'})} \\ &
        RT(SL(c) \tensor_{\mathcal{D}'} d').}
    \]
  \end{enumerate}
\end{lemma}
\begin{proof}
  Claim \itemref{LI:projection-composition:conj}
  follows from
Lemma \ref{L:composition-monoidal}. Claim
  \itemref{LI:projection-composition:proj} follows from the
  definition of the projection morphism \eqref{eq:projection} and
  \itemref{LI:projection-composition:conj}.
\end{proof}
The following lemma contains key compatibilities between composition,
base change, and the projection formula. For the monoidal categories
and their functors arising between morphisms of ringed spaces or
topoi, such results can usually be checked by hand (e.g., \cite[Tag
\spref{0B6B}]{stacks-project}). For unbounded derived categories of
quasi-coherent sheaves, these arguments can fail,
because the relevant functors are obtained from adjoint functor
theorems.
\begin{lemma}\label{L:projection-functorial-bc}
  Consider a $2$-commutative diagram of symmetric monoidal categories:
  \[
    \xymatrix{\mathcal{C}' \ar[r]^{L'}  & \mathcal{D}'\\
      \mathcal{C} \ar[r]^L \ar[u]^{F} \urtwocell\omit{\kappa}& \mathcal{D} \ar[u]_{S}}
  \]
  Assume that $L$, $L'$, $F$, and
  $S$ are strong monoidal and admit respective right adjoints $R$, $R'$, $G$, and
  $T$.
  \begin{enumerate}
  \item \label{LI:projection-functorial-bc:functorial} If
    $c \in \mathcal{C}$ and $d'\in \mathcal{D}'$, then the following
    diagram commutes:
    \[
      \xymatrix@+1pc{c \tensor_{\mathcal{C}} RT(d')
        \ar[d]_{\pi_{c,T(d')}^{L,R}} \ar[r]^{\ID{}\tensor \kappa^\vee} \ar[dr]^-{\pi^{SL,RT}_{c,d'}}
        & c \tensor_{\mathcal{C}} GR'(d') \ar[r]^-{\pi_{c,R'(d')}^{F,G}} \ar[dr]^-{\pi^{L'F,GR'}_{c,d'}}&
        G(F(c) \tensor_{\mathcal{C}'} R'(d')) \ar[d]^{G\pi_{F(c),d'}^{L',R'}} \\
        R(L(c) \tensor_{\mathcal{D}} T(d')) \ar[r]_-{R\pi_{L(c),d'}^{S,T}} &
        RT(SL(c) \tensor_{\mathcal{D'}} d') \ar[r]_{\kappa^\vee(\kappa
          \tensor \ID{})} & GR'(L'F(c) \tensor_{\mathcal{D'}} d')
      }
    \]
  \item \label{LI:projection-functorial-bc:base-change} If
    $d\in \mathcal{D}$, there are natural base change
    morphisms:
    \[
      \beta_d^\kappa, \tilde{\beta}^\kappa_d \colon FR(d) \to R'S(d),
    \]
    which are adjoint to the compositions:
    \begin{align*}
      R(d) &\xrightarrow{R(\eta_d^{S,T})} RTS(d)
      \xrightarrow{\kappa^\vee(S(d))} GR'S(d)\\
      L'FR(d) &\xrightarrow{\kappa(R(d))} SLR(d) \xrightarrow{S(\epsilon^{L,R}_d)} S(d),
    \end{align*}
    respectively, and $\beta_d^\kappa =
    \tilde{\beta}^\kappa_d$. 
  \item \label{LI:projection-functorial-bc:projection} If
    $c\in \mathcal{C}$ and $d\in \mathcal{D}$, then the following
    diagram commutes:
    \[
      \xymatrix@R-1pc{ &  F(c \tensor R(d))
        \ar[r]^-{F(\pi_{c,d}^{L,R})} & FR(L(c) \tensor d)
        \ar[r]^{\beta_{L(c) \tensor d}}& R'S(L(c) \tensor d) \\ F(c) \tensor FR(d)
        \ar[dr]_{\ID{} \tensor \beta_d} \ar[ur]^{\mu^F_{c,R(d)}} & & &  \\ &
        F(c) \tensor R'S(d) \ar[r]_{\pi^{L',R'}_{F(c),S(d)}} &
        R'(L'F(c) \tensor S(d)) \ar[r]^{R'(\kappa \tensor \ID{})} &
        R'(SL(c) \tensor S(d)) \ar[uu]_{R'(\mu^S_{L(c),d})} }
    \]
\end{enumerate}

\end{lemma}
\begin{proof}
  For \itemref{LI:projection-functorial-bc:functorial}, combine Lemma
  \ref{L:projection-composition}\itemref{LI:projection-composition:proj}
  with Remark \ref{R:projection-natural-iso}. Claim
  \itemref{LI:projection-functorial-bc:base-change} follows from the
  commutativity of the following diagram:
  \[
    \xymatrix@C+1.6pc@R+1pc{ FR(d) \ar[r]^-{FR(\eta^{S,T}_d)} \ar[d]_{\eta^{L',R'}_{FR(d)}} & FRTS(d) \ar[r]^-{F(\kappa^\vee)} \ar[d]_-{\eta^{L',R'}_{FRTS(d)}} & FGR'S(d) \ar[d]^-{\eta^{L',R'}_{FGR'S(d)}} \ar[r]^-{\epsilon^{F,G}_{R'S(d)}} & R'S(d) \ar[d]^-{\eta^{L',R'}_{R'S(d)}}\\
      R'L'FR(d) \ar[r]^-{R'L'FR(\eta^{S,T}_d)} \ar[d]_{R'(\kappa)} & R'L'FRTS(d) \ar[d]_{R'(\kappa)} \ar[r]^{R'L'F(\kappa^\vee)} &
      R'L'FGR'S(d) \ar[r]^-{R'L'(\epsilon^{F,G}_{R'S(d)})} \ar[d]^{R'(\kappa)} & R'L'R'S(d) \ar[d]^{R'(\epsilon^{L',R'}_{S(d)})}\\
      R'SLR(d) \ar[d]_{R'S(\epsilon^{L,R}_d)} \ar[r]^-{R'SLR(\eta^{S,T}_d)} & R'SLRTS(d) \ar[r]^-{R'SL(\kappa^\vee)} \ar[d]_{R'S(\epsilon^{L,R}_{TS(d)})} & R'SLGR'S(d) \ar[d]_{R'(\epsilon^{SL,GR'}_{S(d)})} \ar[r]^{R'(\epsilon^{SL,GR'}_{S(d)})} & R'S(d)\\
    R'S(d) \ar[r]_-{R'S(\eta^{S,T}_d)} & R'STS(d) \ar[r] & R'S(d) \ar@{=}[ur]& }
  \]
  together with the observations that the morphism along the top is
  $\beta^\kappa$, the morphism down the left is
  $\tilde{\beta}^\kappa$, and the morphisms on the bottom and right
  are the identity.

  Claim
  \itemref{LI:projection-functorial-bc:projection} follows from the
  commutativity of the following diagram:
  \[
    \xymatrix@C+1.6pc{F(c \tensor R(d)) \ar[r]^-{F(\eta_c^{L,R} \tensor \ID{})} & F(RL(c) \tensor R(d)) \ar[r]^-{F(\nu^{L,R}_{L(c),d})} & FR(L(c) \tensor d) \ar[d]^{FR(\eta^{S,T}_{L(c) \tensor d})}\\
      F(c) \tensor FR(d) \ar[u]^{\mu^F_{c,R(d)}} \ar[r]^{F(\eta_c^{L,R}) \tensor \ID{}} \ar[d]_{\ID{}\tensor FR(\eta^{S,T}_d)} & FRL(c) \tensor FR(d) \ar[d]_{FR(\eta^{S,T}_{L(c)}) \tensor FR(\eta^{S,T}_d)} \ar[u]^{\mu^F_{RL(c),R(d)}} \ar[ur]_{\mu^{FR}_{L(c),d}} & FRTS(L(c) \tensor d) \ar[d]^{F(\kappa^\vee)}\\
      F(c) \tensor FRTS(d) \ar[r]^-{F(\eta^{SL,RT}_c) \tensor \ID{}} \ar[d]_{\ID{}\tensor F(\kappa^\vee)} & FRTSL(c) \tensor FRTS(d) \ar[ur]_{\mu^{FRTS}_{L(c),d}} \ar[d]_{F(\kappa^\vee) \tensor F(\kappa^\vee)} & FGR'S(L(c) \tensor d) \ar[d]^{\epsilon^{F,G}_{R'S(L(c) \tensor d)}} \\
      F(c) \tensor FGR'S(d)  \ar[d]_{\ID{} \tensor \epsilon^{F,G}_{R'S(d)}} & FGR'SL(c) \tensor FGR'S(d) \ar[ur]_{\mu^{FGR'}_{SL(c),S(d)}} \ar[d]_{\epsilon^{F,G}_{R'SL(c)} \tensor \epsilon^{F,G}_{R'S(d)}} & R'S(L(c) \tensor d)\\ F(c) \tensor R'S(d)  \ar[dr]_{\eta^{L',R'}_{F(c)} \tensor \ID{}} & R'SL(c) \tensor R'S(d) \ar[ur]_{\mu^{R'S}_{L(c),d}}  & R'(SL(c) \tensor S(d)) \ar[u]_{R'(\mu^S_{L(c),d})}\\
    & R'L'F(c) \tensor R'S(d) \ar[r]_{\nu^{L',R'}_{F(c), S(d)}} \ar[u]_{R'(\kappa) \tensor R'(\kappa)} & R'(L'F(c) \tensor S(d)) \ar[u]_{R'(\kappa \tensor \ID{})}}
\]
The commutativity of the squares on the right is just Lemmas
\ref{L:composition-monoidal} and \itemref{LI:projection-composition:conj}.
\end{proof}
An object $c\in \mathcal{C}$ is \emph{dualizable} if there is a triple
$(c^*,s,t)$, where $c^* \in \mathcal{C}$ and
$s \colon 1 \to c \tensor c^*$ and $t \colon c^* \tensor c \to 1$ are
morphisms such that the two compositions
\[
  \xymatrix{c \ar[r]^-{\lambda_c^{-1}}  & 1 \tensor c  \ar[r]^-{s \tensor \ID{}} & (c \tensor c^*) \tensor c \ar[r]^-{\alpha_{c,c^*,c}} & c \tensor (c^* \tensor c) \ar[r]^-{\ID{} \tensor t} & c \tensor 1_{\mathcal{C}} \ar[r]^-{\rho_c} & c,\\
    c^* \ar[r]^-{\rho_{c^*}^{-1}} & c^* \tensor 1 \ar[r]^-{\ID{} \tensor
      s} & c^* \tensor (c \tensor c^*) \ar[r]^-{\alpha_{c^*,c,c^*}} & (c^*
    \tensor c) \tensor c^* \ar[r]^-{t \tensor \ID{}} & 1 \tensor c^*
  \ar[r]^-{\lambda_{c^*}} & c^*}
\]
are the identity morphism. Another way of expressing this is that the
functor $c^* \tensor - $ is left adjoint to $c\tensor -$. In the
following standard lemma, we do not require the existence of a right
adjoint $R$ to $L$.
\begin{lemma}\label{L:pb-dualizable}
  If $(L,\mu,\iota)$ is strong monoidal and $c$ is dualizable, then $L(c)$ is dualizable. More precisely: let
  $(c^*,s,t)$ be a dual of $c$. Then $(L(c^*),s_L,t_L)$, where $s_L$ is the
  composition:
  \[
    1_{\mathcal{D}} \xrightarrow{\iota} L(1_{\mathcal{C}})
    \xrightarrow{L(s)} L(c \tensor_{\mathcal{C}} c^*)
    \xrightarrow{\mu_{c,c^*}^{-1}} L(c) \tensor_{\mathcal{D}} L(c^*),
  \]
  and $t_L$ is the composition:
  \[
    L(c^*) \tensor_{\mathcal{D}} L(c) \xrightarrow{\mu_{c^*,c}} L(c^*\tensor_{\mathcal{C}} c) \xrightarrow{L(t)} L(1_{\mathcal{C}}) \xrightarrow{\iota^{-1}} 1_{\mathcal{D}}. 
  \]
  is dual to $L(c)$.
\end{lemma}
\begin{proof}
  This is a routine diagram chase.
\end{proof}
We now come to the main result of this appendix.
\begin{theorem}\label{AT:projection-formula}
  If $c\in \mathcal{C}$ is dualizable, then
  \[
    \pi_{c,d} \colon c \tensor_{\mathcal{C}} R(d) \to R(L(c) \tensor_{\mathcal{D}} d)
  \]
  is an isomorphism.
\end{theorem}
It is not difficult to prove that $c\tensor_{\mathcal{C}} R(d)$ and
$R(L(c) \tensor_{\mathcal{D}} d)$ are isomorphic when $c$ is
dualizable. The subtlety is showing that this isomorphism can be
witnessed by the projection morphism $\pi_{c,d}$. In applications, this is critical.

The standard reference for Theorem \ref{AT:projection-formula} (in the
context of \emph{closed} symmetric monoidal categories) is
\cite[Prop.~3.12]{MR1988072}. Note that Theorem
\ref{AT:projection-formula} is not actually proved in
\loccit{}---there is an extra coherence condition for strong monoidal
functors specified in \cite[Eq.~3.7]{MR1988072}. It is shown in
\cite[Rem.~2.2.10]{MR2271789}, however, that this coherence condition
is implied by the other conditions. Because of its importance to this
article, we give a self-contained proof here using dualizables.
\begin{proof}[Proof of Theorem \ref{AT:projection-formula}]
  Let $(c^*,s,t)$ be a dual of $c$.  Let $x\in \mathcal{C}$.  Then
  observe that we have the following natural sequence of bijections:
  \begin{align*}
    \Hom_{\mathcal{C}}(x,c \tensor_{\mathcal{C}} R(d))
    &\cong \Hom_{\mathcal{C}}(c^* \tensor_{\mathcal{C}} x, R(d))\\
    &\cong \Hom_{\mathcal{D}}(L(c^* \tensor_{\mathcal{C}} x),d)\\
    &\cong \Hom_{\mathcal{D}}(L(c^*)\tensor_{\mathcal{D}} L(x),d)\\
    &\cong \Hom_{\mathcal{D}}(L(x),L(c) \tensor_{\mathcal{D}} d) & \mbox{(Lemma \ref{L:pb-dualizable})}\\
    &\cong \Hom_{\mathcal{C}}(x,R(L(c)\tensor_{\mathcal{D}} d)).
  \end{align*}
  By the Yoneda lemma, it follows that there is a unique isomorphism
  \[
    \pi'_{c,d} \colon c \tensor_{\mathcal{c}} R(d) \simeq R(L(c)
    \tensor_{\mathcal{D}} d)
  \]
  inducing the above. By Lemma \ref{L:projections-equal}, it remains
  to prove that $\tilde{\pi}_{c,d}=\pi'_{c,d}$. We will do this using
  the Yoneda lemma. Fix $f \colon x \to c \tensor_{\mathcal{C}}
  R(d)$. By definition, the $L$-$R$ adjoint to the composition $\tilde{\pi}_{c,d} \circ f$ we
  can express as the composition:
  \[
    L(x) \xrightarrow{L(f)} L(c \tensor_{\mathcal{C}} R(d))
    \xrightarrow{\mu_{c,R(d)}^{-1}} L(c) \tensor_{\mathcal{D}} LR(d)
    \xrightarrow{\ID{} \tensor \epsilon_d} L(c) \tensor_{\mathcal{D}}
    d.
  \]
  The adjoint to this morphism (afforded by $L(c^*)\tensor-$ and
  $L(c)\tensor-$) is thus the composition: 
  \begin{align*}
    L(c^*) \tensor_{\mathcal{D}} L(x) &\xrightarrow{\ID{}\tensor L(f)} L(c^*) \tensor_{\mathcal{D}} L(c \tensor_{\mathcal{C}} R(d))\\
                                      &\xrightarrow{\ID{} \tensor \mu_{c,R(d)}^{-1}} L(c^*) \tensor_{\mathcal{D}} (L(c) \tensor_{\mathcal{D}} LR(d))\\
    &\xrightarrow{\alpha_{\mathcal{D},L(c^*),L(c), LR(d)}} (L(c^*) \tensor_{\mathcal{D}} L(c)) \tensor_{\mathcal{D}} LR(d) \\
    &\xrightarrow{\ID{} \tensor \epsilon_d} (L(c^*) \tensor_{\mathcal{D}} L(c)) \tensor_{\mathcal{D}} d \xrightarrow{t_L \tensor \ID{}} 1_{\mathcal{D}} \tensor_{\mathcal{D}} d \xrightarrow{\lambda_{\mathcal{D},d}} d.
  \end{align*}
  We now look at the image of
  $f$ under the compositions defining $\pi'_{c,d}$ via the
  Yoneda lemma.  What we see is that:
  \begin{align*}
    f &\mapsto (c^* \tensor_{\mathcal{C}} x \xrightarrow{\ID{} \tensor f} c^* \tensor_{\mathcal{C}} (c \tensor_{\mathcal{C}} R(d)) \xrightarrow{\alpha_{\mathcal{C},c^*,c,R(d)}} (c^* \tensor_{\mathcal{C}} c) \tensor_{\mathcal{D}} R(d)\\
      &\qquad \xrightarrow{t \tensor \ID{}} 1_{\mathcal{C}} \tensor_{\mathcal{C}} R(d) \xrightarrow{\lambda_{\mathcal{C},R(d)}} R(d))\\
      &\mapsto (L(c^* \tensor_{\mathcal{C}} x)\xrightarrow{L(\ID{} \tensor f)} L(c^* \tensor_{\mathcal{C}} (c \tensor_{\mathcal{C}} R(d))) \\
      &\qquad \xrightarrow{L(a_{\mathcal{C},c^*,c,R(d)})} L((c^*\tensor_{\mathcal{C}} c) \tensor_{\mathcal{C}} R(d)) \xrightarrow{L(t \tensor \ID{})} L(1_{\mathcal{C}} \tensor_{\mathcal{C}} R(d)) \\
      &\qquad \xrightarrow{L(\lambda_{R(d)})} LR(d) \xrightarrow{\epsilon_d} d)
  \end{align*}
  It remains to show that precomposing the above morphism with
  $\mu_{c^*,x}$ coincides with the other morphism described
  above. This follows from the commutativity of the following diagram,
  and that all of the vertical arrows are
  isomorphisms:
  \[
    \tiny{\xymatrix@C-0.8pc@R+1pc{L(c^* \tensor x) \ar[r]^-{\ID{}\tensor L(f)} \ar[d]_{\mu_{c^*,x}^{-1}} & L(c^* \tensor (c\tensor Rd)) \ar[r]^-{\ID{}\tensor \mu_{c,Rd}^{-1}} \ar[d]_{\mu_{c^*,c\tensor Rd}^{-1}}  & L((c^*\tensor c) \tensor Rd) \ar[r]^-{\alpha_{\mathcal{D}}} & L(c^*\tensor c) \tensor LRd \ar[r]^{L(t)\tensor \ID{}} \ar[d]_{\mu_{c^*,c}^{-1} \tensor \ID{}} & L(1_{\mathcal{C}}) \tensor LRd \ar[d]^{\iota \tensor \ID{}}\\
        L(c^*) \tensor L(x) \ar[r]_-{\ID{} \tensor L(f)} & L(c^*)
        \tensor L(c \tensor Rd) \ar[r]^-{\ID \tensor \mu_{c,Rd}^{-1}} &
        L(c^*) \tensor (Lc \tensor LRd) \ar[r]^{L(\alpha_{\mathcal{C}})} &
        (L(c^*) \tensor Lc) \tensor LRd \ar[r]^{t_L \tensor \ID{}} &
        1_{\mathcal{D}} \tensor LRd.}}
  \]
\end{proof}

\section{Two lemmas for ringed topoi}
We include in this appendix two simple lemmas, which we expect to be
well-known to experts.
\begin{lemma}\label{L:projection-topoi-affine}
  Let $\mathcal{W}$ be a ringed topos. Let $\mathcal{B}$ be a sheaf of
  $\Orb_{\mathcal{W}}$-algebras. Let $\mathcal{W}'$ be the ringed
  topos $(\mathcal{W},\mathcal{B})$. There is an induced morphism of ringed topoi
  $j \colon \mathcal{W}' \to \mathcal{W}$. Let
  $\mathcal{M} \in \DCAT(\mathcal{W})$ and
  $\mathcal{N} \in \DCAT(\mathcal{W}')$. Then the projection morphism
  \[
    \pi_{\mathcal{M},\mathcal{N}} \colon \mathcal{M}
    \tensor^{\LDERF}_{\Orb_{\mathcal{W}}} \RDERF j_*\mathcal{N} \to
    \RDERF j_*(\LDERF j^*\mathcal{M} \tensor^{\LDERF}_{\Orb_{\mathcal{W'}}}
    \mathcal{N})
  \]
  is an isomorphism. In particular, if
  $\mathcal{Q} \in \DPCOH{-}(\mathcal{W})$ and
  $\mathcal{P} \in \langle \RDERF j_*\DPCOH{-}(\mathcal{W}')\rangle$,
  then
  $\mathcal{Q} \tensor^{\LDERF}_{\Orb_{\mathcal{W}}} \mathcal{P} \in
  \langle \RDERF j_*\DPCOH{-}(\mathcal{W}')\rangle$.
\end{lemma}
\begin{proof}
  Let $\mathcal{F}$ be a K-flat complex of
  $\Orb_{\mathcal{W}}$-modules quasi-isomorphic to $\mathcal{M}$ and $\mathcal{P}$ a
  K-flat complex of $\Orb_{\mathcal{W}'}$-modules quasi-isomorphic to
  $\mathcal{N}$. The exactness of $j_*$ implies that
  $\mathcal{M} \otimes^{\LDERF}_{\Orb_{\mathcal{W}}} \RDERF j_*\mathcal{N}$ is the total
  complex of 
  $(\mathcal{F}^r \otimes_{\Orb_{\mathcal{W}}} j_*\mathcal{P}^s)_{r,s}$. Clearly,
  \[
    \mathcal{F}^r \otimes_{\Orb_{\mathcal{W}}} j_*\mathcal{P}^s = j_*(j^*\mathcal{F}^r \otimes_{\Orb_{\mathcal{W}'}} \mathcal{P}^s).
    \]
    Moreover, $j_*$ commutes with the formation of total complexes (it
    commutes with small coproducts). The result is now immediate. For
    the latter claim, we set
    \[
      \DCAT_{\mathcal{Q}} = \{ \mathcal{R} \in \langle \RDERF
      j_*\DPCOH{-}(\mathcal{W}')\rangle \suchthat \mathcal{Q}
      \tensor^{\LDERF} \mathcal{R} \in \langle \RDERF
      j_*\DPCOH{-}(\mathcal{W}')\rangle\}.
    \]
    Clearly, $\DCAT_{\mathcal{Q}}$ is a thick triangulated subcategory
    of $\langle \RDERF j_*\DPCOH{-}(\mathcal{W}')\rangle$. Thus, 
    it suffices to prove that 
    $\mathcal{N} \in \DPCOH{-}(\mathcal{W}')$ implies
    $\mathcal{Q} \tensor^{\LDERF}_{\Orb_{\mathcal{W}}} \RDERF
    j_*\mathcal{N} \in \langle \RDERF
    j_*\DPCOH{-}(\mathcal{W}')\rangle$. This is obvious from the
    projection formula.
\end{proof}
\begin{lemma}\label{L:tor-ind-topoi}
  Let $\pi \colon \mathcal{Y} \to \mathcal{W}$ be a morphism of ringed
  topoi. Let $\mathcal{B}$ be a sheaf of
  $\Orb_{\mathcal{W}}$-algebras. Let $\mathcal{W}'$ and $\mathcal{Y}'$
  be the ringed topoi $(\mathcal{W},\mathcal{B})$ and
  $(\mathcal{Y},\pi^*\mathcal{B})$, respectively. There is an induced
  $2$-commutative diagram of ringed topoi:
  \[
    \xymatrix{\mathcal{Y}' \ar[r]^{\pi'} \ar[d]_{j'} &
      \mathcal{W}' \ar[d]^j \\ \mathcal{Y} \ar[r]^\pi & \mathcal{W}.}
  \]
  If $\pi$ and $j$ are tor-independent and 
  $\mathcal{N} \in \DCAT(\mathcal{W}')$, then there is a natural
  isomorphism:
  \[
    \LDERF \pi^*\RDERF j_*\mathcal{N} \simeq \RDERF
    j'_*\LDERF \pi'^*\mathcal{N}.
  \]
  In particular, if
  $\mathcal{Q} \in \thick{\RDERF j_*\DPCOH{-}(\mathcal{W}')}$, then
  $\LDERF \pi^*\mathcal{Q} \in \thick{\RDERF j'_*\DPCOH{-}(\mathcal{Y}')}$.
\end{lemma}
\begin{proof}
  Now $j_*$, $j'_{*}$ are exact and
  $\pi^{-1}j_{*} = j'_{*}\pi'^{-1}$. By tor-independence of
  $\pi$ and $j$:
  \begin{align*}
    \LDERF \pi^*\RDERF j_{*}\mathcal{N} &= \Orb_{\mathcal{Y}} \otimes_{\pi^{-1}\Orb_{\mathcal{W}}}^{\LDERF} \pi^{-1}j_{*}\mathcal{N}\simeq (\Orb_{\mathcal{W}} \otimes_{\pi^{-1}\Orb_{\mathcal{Y}}}^{\LDERF} \pi^{-1}j_*\Orb_{\mathcal{Y}'}) \otimes^{\LDERF}_{\pi^{-1}j_*\Orb_{\mathcal{Y}'}}\pi^{-1}j_{*}\mathcal{N}\\
                       &\simeq j_{*}'\Orb_{\mathcal{Y}'} \otimes^{\LDERF}_{\pi^{-1}j_*\Orb_{\mathcal{Y}'}}\pi^{-1}j_{*}\mathcal{N}\\
                       &\simeq j_{*}'\Orb_{\mathcal{Y}'} \otimes^{\LDERF}_{j_{*}'\pi'^{-1}\Orb_{\mathcal{W}'}} j_{*}'\pi'^{-1}\mathcal{N}\\
                       &\simeq j_{*}'(\Orb_{\mathcal{Y}'} \otimes^{\LDERF}_{\pi'^{-1}\Orb_{\mathcal{W}'}} \pi'^{-1}\mathcal{N})= \RDERF j_{*}'\LDERF \pi'^*\mathcal{N}.
  \end{align*}
  The latter claim follows from a similar argument to that
  in
  Lemma \ref{L:projection-topoi-affine}.
\end{proof}
 \bibliographystyle{bibtex_db/dary}
\bibliography{bibtex_db/references}

\providecommand{\MR}{\relax\ifhmode\unskip\space\fi MR }
\providecommand{\MRhref}[2]{%
  \href{http://www.ams.org/mathscinet-getitem?mr=#1}{#2}
}
\providecommand{\href}[2]{#2}
\begin{thebibliography}{BZNP17}

\bibitem[AT19]{abramovich_temkin_factorization_qe_char0}
D.~Abramovich and M.~Temkin, \emph{Functorial factorization of birational maps
  for qe schemes in characteristic 0}, Algebra Number Theory \textbf{13}
  (2019), no.~2, 379--424.

\bibitem[Bal11]{MR2793026}
M.~Ballard, \emph{Derived categories of sheaves on singular schemes with an
  application to reconstruction}, Adv. Math. \textbf{227} (2011), no.~2,
  895--919.

\bibitem[BB03]{MR1996800}
A.~Bondal and M.~Van~den Bergh, \emph{Generators and representability of
  functors in commutative and noncommutative geometry}, Mosc. Math. J.
  \textbf{3} (2003), no.~1, 1--36, 258.

\bibitem[BJ14]{MR3250065}
B.~Bhatt and A.~J. de~Jong, \emph{Lefschetz for local {P}icard groups}, Ann.
  Sci. \'{E}c. Norm. Sup\'{e}r. (4) \textbf{47} (2014), no.~4, 833--849.

\bibitem[BZNP17]{BZNP_integral}
D.~Ben-Zvi, D.~Nadler, and A.~Preygel, \emph{Integral transforms for coherent
  sheaves}, J. Eur. Math. Soc. (JEMS) \textbf{19} (2017), no.~12, 3763--3812.

\bibitem[CLO12]{MR2979821}
B.~Conrad, M.~Lieblich, and M.~Olsson, \emph{Nagata compactification for
  algebraic spaces}, J. Inst. Math. Jussieu \textbf{11} (2012), no.~4,
  747--814.

\bibitem[Con06]{MR2266885}
B.~Conrad, \emph{Relative ampleness in rigid geometry}, Ann. Inst. Fourier
  (Grenoble) \textbf{56} (2006), no.~4, 1049--1126.

\bibitem[CS53]{cartan-serre_finiteness}
H.~Cartan and J.-P. Serre, \emph{Un th\'{e}or\`eme de finitude concernant les
  vari\'{e}t\'{e}s analytiques compactes}, C. R. Acad. Sci. Paris \textbf{237}
  (1953), 128--130.

\bibitem[CT09]{MR2524597}
B.~Conrad and M.~Temkin, \emph{Non-{A}rchimedean analytification of algebraic
  spaces}, J. Algebraic Geom. \textbf{18} (2009), no.~4, 731--788.

\bibitem[Dev20]{devadas-phd-henselian}
S.~Devadas, \emph{Morphisms and {C}ohomological {C}omparison for {H}enselian
  {S}chemes}, ProQuest LLC, Ann Arbor, MI, 2020, Thesis (Ph.D.)--Stanford
  University.

\bibitem[Duc15]{MR3330765}
A.~Ducros, \emph{Cohomological finiteness of proper morphisms in algebraic
  geometry: a purely transcendental proof, without projective tools}, Berkovich
  spaces and applications, Lecture Notes in Math., vol. 2119, Springer, Cham,
  2015, pp.~135--140.

\bibitem[EGA]{EGA}
A.~Grothendieck, \emph{\'{E}l\'ements de g\'eom\'etrie alg\'ebrique}, I.H.E.S.
  Publ. Math. \textbf{4, 8, 11, 17, 20, 24, 28, 32} (1960, 1961, 1961, 1963,
  1964, 1965, 1966, 1967).

\bibitem[Eps66]{MR213412}
D.~B.~A. Epstein, \emph{Functors between tensored categories}, Invent. Math.
  \textbf{1} (1966), 221--228.

\bibitem[FHM03]{MR1988072}
H.~Fausk, P.~Hu, and J.~P. May, \emph{Isomorphisms between left and right
  adjoints}, Theory Appl. Categ. \textbf{11} (2003), No. 4, 107--131.

\bibitem[FK18]{fujiwara-kato-I}
K.~Fujiwara and F.~Kato, \emph{Foundations of rigid geometry. {I}}, EMS
  Monographs in Mathematics, European Mathematical Society (EMS), Z\"{u}rich,
  2018.

\bibitem[FL85]{MR801033}
W.~Fulton and S.~Lang, \emph{Riemann-{R}och algebra}, Grundlehren der
  Mathematischen Wissenschaften [Fundamental Principles of Mathematical
  Sciences], vol. 277, Springer-Verlag, New York, 1985.

\bibitem[GAGA]{MR0082175}
J.~P. Serre, \emph{G\'eom\'etrie alg\'ebrique et g\'eom\'etrie analytique},
  Ann. Inst. Fourier, Grenoble \textbf{6} (1955--1956), 1--42.

\bibitem[GR84]{MR755331}
H.~Grauert and R.~Remmert, \emph{Coherent analytic sheaves}, Grundlehren der
  Mathematischen Wissenschaften [Fundamental Principles of Mathematical
  Sciences], vol. 265, Springer-Verlag, Berlin, 1984.

\bibitem[Gro57]{MR0102537}
A.~Grothendieck, \emph{Sur quelques points d'alg\`ebre homologique}, T\^{o}hoku
  Math. J. (2) \textbf{9} (1957), 119--221.

\bibitem[HR17a]{perfect_complexes_stacks}
J.~Hall and D.~Rydh, \emph{Perfect complexes on algebraic stacks}, Compositio
  Math. \textbf{153} (2017), no.~11, 2318--2367.

\bibitem[HR17b]{telescope-stacks}
J.~Hall and D.~Rydh, \emph{The telescope conjecture for algebraic stacks}, J.
  Topol. \textbf{10} (2017), no.~3, 776--794.

\bibitem[HR22]{mayer-vietoris}
J.~Hall and D.~Rydh, \emph{{{M}ayer--{V}ietoris squares in algebraic
  geometry}}, J. London Math. Soc. (2) (2022), to appear.

\bibitem[Kie67]{MR0210948}
R.~Kiehl, \emph{Der {E}ndlichkeitssatz f{\"u}r eigentliche {A}bbildungen in der
  nichtarchimedischen {F}unktionentheorie}, Invent. Math. \textbf{2} (1967),
  191--214.

\bibitem[Kie72]{MR0382280}
R.~Kiehl, \emph{Ein ``{D}escente''-{L}emma und {G}rothendiecks
  {P}rojektionssatz f\"ur nichtnoethersche {S}chemata}, Math. Ann. \textbf{198}
  (1972), 287--316.

\bibitem[Knu71]{MR0302647}
D.~Knutson, \emph{Algebraic spaces}, Lecture Notes in Mathematics, Vol. 203,
  Springer-Verlag, Berlin, 1971.

\bibitem[K{\"o}p74]{kopf_rigid_gaga}
U.~K{\"o}pf, \emph{{{\"U}ber eigentliche Familien algebraischer Varietäten
  {\"u}ber affinoiden R{\"a}umen.}}, Ph.D. thesis, {Schr. Math. Inst. Univ.
  M{\"u}nster}, 1974.

\bibitem[Lip09]{MR2490557}
J.~Lipman, \emph{Notes on derived functors and {G}rothendieck duality},
  Foundations of {G}rothendieck duality for diagrams of schemes, Lecture Notes
  in Math., vol. 1960, Springer, Berlin, 2009, pp.~1--259.

\bibitem[LN07]{MR2346195}
J.~Lipman and A.~Neeman, \emph{Quasi-perfect scheme-maps and boundedness of the
  twisted inverse image functor}, Illinois J. Math. \textbf{51} (2007), no.~1,
  209--236.

\bibitem[LO08]{MR2434692}
Y.~Laszlo and M.~Olsson, \emph{The six operations for sheaves on {A}rtin
  stacks. {I}. {F}inite coefficients}, Publ. Math. Inst. Hautes \'Etudes Sci.
  \textbf{107} (2008), 109--168.

\bibitem[ML63]{MR0170925}
S.~Mac~Lane, \emph{Natural associativity and commutativity}, Rice Univ. Studies
  \textbf{49} (1963), no.~4, 28--46.

\bibitem[ML98]{MR1712872}
S.~Mac~Lane, \emph{Categories for the working mathematician}, second ed.,
  Graduate Texts in Mathematics, vol.~5, Springer-Verlag, New York, 1998.

\bibitem[MS06]{MR2271789}
J.~P. May and J.~Sigurdsson, \emph{Parametrized homotopy theory}, Mathematical
  Surveys and Monographs, vol. 132, American Mathematical Society, Providence,
  RI, 2006.

\bibitem[Nee96]{MR1308405}
A.~Neeman, \emph{The {G}rothendieck duality theorem via {B}ousfield's
  techniques and {B}rown representability}, J. Amer. Math. Soc. \textbf{9}
  (1996), no.~1, 205--236.

\bibitem[Nee01]{MR1812507}
A.~Neeman, \emph{Triangulated categories}, Annals of Mathematics Studies, vol.
  148, Princeton University Press, Princeton, NJ, 2001.

\bibitem[{Nee}18]{Neeman_Approx}
A.~{Neeman}, \emph{{Triangulated categories with a single compact generator and
  a Brown representability theorem}}, April 2018,
  \href{http://arXiv.org/abs/1804.02240}{\mbox{arXiv:1804.02240}}.

\bibitem[Oka50]{oka_coherence}
K.~Oka, \emph{Sur les fonctions analytiques de plusieurs variables. {VII}.
  {S}ur quelques notions arithm\'{e}tiques}, Bull. Soc. Math. France
  \textbf{78} (1950), 1--27.

\bibitem[Ryd14]{MO_coh_proper}
D.~Rydh, \emph{If the direct image of f preserves coherent sheaves on
  noetherian schemes, how to show f is proper?}, MathOverflow, October 2014,
  \url{https://mathoverflow.net/q/182902}.

\bibitem[Ryd15]{rydh-2009}
D.~Rydh, \emph{Noetherian approximation of algebraic spaces and stacks}, J.
  Algebra \textbf{422} (2015), 105--147.

\bibitem[SAG]{lurie_sag}
J.~Lurie, \emph{{S}pectral {A}lgebraic {G}eometry}, available on homepage, Oct
  2016.

\bibitem[Sch92]{MR1179103}
C.~Scheiderer, \emph{Quasi-augmented simplicial spaces, with an application to
  cohomological dimension}, J. Pure Appl. Algebra \textbf{81} (1992), no.~3,
  293--311.

\bibitem[SGA1]{SGA1}
A.~Grothendieck and M.~Raynaud, \emph{Rev{\^e}tements {\'e}tales et groupe
  fondamental}, S{\'e}minaire de G{\'e}om{\'e}trie Alg{\'e}brique, I.H.E.S.,
  1963.

\bibitem[SGA2]{SGA2}
A.~Grothendieck, \emph{Cohomologie locale des faisceaux coh\'{e}rents et
  th\'{e}or\`emes de {L}efschetz locaux et globaux}, North-Holland Publishing
  Co., Amsterdam; Masson \& Cie, \'{E}diteur, Paris, 1968.

\bibitem[SGA6]{MR0354655}
\emph{Th\'eorie des intersections et th\'eor\`eme de {R}iemann-{R}och}, Lecture
  Notes in Mathematics, Vol. 225, Springer-Verlag, Berlin, 1971, S{\'e}minaire
  de G{\'e}om{\'e}trie Alg{\'e}brique du Bois-Marie 1966--1967 (SGA 6),
  Dirig{\'e} par P. Berthelot, A. Grothendieck et L. Illusie. Avec la
  collaboration de D. Ferrand, J. P. Jouanolou, O. Jussila, S. Kleiman, M.
  Raynaud et J. P. Serre.

\bibitem[SR72]{MR0338002}
N.~Saavedra~Rivano, \emph{Cat\'{e}gories {T}annakiennes}, Lecture Notes in
  Mathematics, Vol. 265, Springer-Verlag, Berlin-New York, 1972.

\bibitem[Stacks]{stacks-project}
The {Stacks Project Authors}, \emph{{S}tacks {P}roject},
  \url{http://stacks.math.columbia.edu}.

\bibitem[Zav21]{zavyalov2021coherent}
B.~Zavyalov, \emph{Almost coherent modules and almost coherent sheaves}, 2021,
  \href{http://arXiv.org/abs/2110.10773}{\mbox{arXiv:2110.10773}}.

\end{thebibliography}
\end{document}